\documentclass[11pt,draft]{article}

\pdfoutput=1

\RequirePackage{my_packages}
\RequirePackage{my_theorems_en}
\RequirePackage{my_macros}
\RequirePackage{my_tikz}

\usepackage[backend = biber,       
            language = english,
            style    = numeric-comp,   
            giveninits = true,
            isbn = false,
            doi = false,
            sorting = nyt,
            backref=true,
            sortcites
            ]{biblatex}

\DeclareNameAlias{sortname}{last-first}
\renewbibmacro{in:}{%
  \ifentrytype{article}{}{%
  \printtext{\bibstring{in}\intitlepunct}}}
\AtEveryBibitem{
    \clearfield{urlyear}
    \clearfield{urlmonth}
}

\AtEveryBibitem{%
  \clearfield{note}%
  \clearfield{url}
}

\DeclareNameAlias{labelname}{given-family}
\renewrobustcmd*{\bibinitdelim}{\,}

\title{Stallings automata for free-times-abelian groups: intersections and index
}
\author{Jordi Delgado$^*$ and Enric Ventura$^\dag$\\[7pt]
\emph{Dedicated to the memory of our late colleague and friend Paul Schupp (1937--2022)}}
\date{\vspace{-7pt}
    $^*$Departamento de Matemáticas, Universidad del País Vasco\\%
    $^\dag$Departament de Matemàtiques, Universitat Politècnica de Catalunya\\[3ex]%
    \today
}

\newcommand{\Addresses}{{
  \bigskip
  \footnotesize

  Jordi Delgado\\\nopagebreak
  \textsc{Departamento de Matemáticas, Universidad del País Vasco, España}\\\nopagebreak
  \url{jdelgado@crm.cat}

  \medskip

  Enric Ventura\\\nopagebreak
  \textsc{Departament de Matemàtiques, Universitat Politècnica de Catalunya, Spain}\\\nopagebreak
  \url{enric.ventura@upc.edu }

}}

\hypersetup{final}  

\begin{document}

\maketitle

\begin{abstract}
\noindent We extend the classical Stallings theory (describing subgroups of free groups as automata) to direct products of free and abelian groups: after introducing \defin{enriched automata} (\ie automata with extra abelian labels), we obtain an explicit bijection between subgroups and a certain type of such enriched automata, which
--- as it happens in the free group ---
is computable in the finitely generated case.

\medskip
\noindent
This approach provides a neat geometric description of (even non finitely generated) intersections of finitely generated subgroups within this non-Howson family. In particular, we give a geometric solution to the subgroup intersection problem and the finite index problem, providing recursive bases and transversals, respectively.
\end{abstract}

\bigskip



\vspace{15pt}
\section{Introduction}

Stallings automata constitute the main modern tool to understand and work with the lattice of subgroups of a free group $\Free[X]$ (usually assumed to have finite rank $n = \card X$, denoted by $\Fn$). In the seminal paper~\cite{StallingsTopologyFiniteGraphs1983}, \citeauthor{StallingsTopologyFiniteGraphs1983} used a topological approach to construct a natural and algorithmic-friendly bijection $H\leftrightarrow \stallings{H,X}$ between subgroups of $\Free[X]$ and certain kind of $X$-automata;
see~\cite{KapovichStallingsFoldingsSubgroups2002, BartholdiRationalSubsetsGroups2010,
Delgado_list_2022}
for a more combinatorial approach closer to ours. This bijection has proved to be very useful to obtain modern solutions to both classical and new problems regarding subgroups of the free group. First easy applications are
computability of bases of finitely generated subgroups,
the solution to the membership problem in $\Fn$,
and description of finite index subgroups. Specially relevant to us is the product (or pull-back) technique, which makes it possible to construct $\stallings{H\cap K,X}$  from $\stallings{H,X}$ and $\stallings{K,X}$; this immediately implies that $\Fn$ is Howson, and allows to describe (compute a basis of) $H\cap K$ in terms of (bases of) $H$ and $K$. More recent applications of this fruitful theory are included in~\cite{
MiasnikovAlgebraicExtensionsFree2007,
SilvaAlgorithmDecideWhether2008,
PuderPrimitiveWordsFree2014,
DelgadoLatticeSubgroupsFree2020,
MargolisClosedSubgroupsProV2001,
RoigComplexityWhiteheadMinimization2007,
VenturaFixedSubgroupsMaximal1997}.
See \cite{Delgado_list_2022} for a recent survey on applications of Stallings automata.

This geometric approach constituted the seed for many successful attempts at generalization; see for example \cite{
KharlampovichStallingsGraphsQuasiconvex2017,
KapovichFoldingsGraphsGroups2005,
IvanovIntersectionFinitelyGenerated1999,
IvanovIntersectingFreeSubgroups2001,
DelgadoIntersectionProblemDroms2018,
SilvaFiniteAutomataSchreier2016}. In this paper, we present a generalization into another direction, namely to direct products of finitely generated free and abelian groups, \ie groups of the form $\FTA = \Fta$. This is part of a more ambitious project started in \cite[Chapter~5]{DelgadoExtensionsFreeGroups2017}, continued in \cite{delgado_Stallings_ABF}, and aiming at the much more general class of semidirect extensions of free groups, \ie groups of the form $\Fn \ltimes_A G$. We restrict ourselves to free-times-abelian groups (\ie $G$ finitely generated abelian and $A$ trivial). The more involved theory for the general semidirect scenario is in progress and will appear published in the near future; see~\cite{delgado_Stallings_free_ext}.
The theory of intersections of free-times-abelian groups, in turn, have a natural continuation in \cite{delgado_configurations_2021} (where the possible configurations of multiple intersections are studied) and \cite{delgado_universal_quotients} (where the results in \cite{delgado_configurations_2021} are used to obtain groups which, in some sense, admit a structure of quotients as complicated as possible).

So, we revisit the family of free-times-abelian groups (already considered by the same authors in \cite{DelgadoAlgorithmicProblemsFreeabelian2013})
now from a geometric point of view.
This approach provides new insight into the properties and behavior of subgroups that refines and clarifies some known results in the finitely generated realm,  and extends into the non finitely generated one.
The main idea is to suitably enrich classical Stallings automata with abelian labels
to make them expressive enough to represent every subgroup of $
\FTA$,
and flexible enough to make this representation unique (and algorithmic when restricted to finitely generated subgroups).
With this bijection at hand, we
interpret the notion of basis (for subgroups of $\FTA$), and geometrically rephrase the solution to the membership problem
 $\SMP(\FTA)$
given in~\cite{DelgadoAlgorithmicProblemsFreeabelian2013}. Then, we go on to analyze intersections; note that this must be more complicated than just computing products (of the corresponding enriched automata) since $\FTA$ is not a Howson group in general, whereas products of finite objects are again finite. Our approach makes it possible to geometrically understand arbitrary intersections of subgroups of $\FTA$ as (a certain technical variation of) Cayley digraphs of abelian groups. Moreover, when the intersecting subgroups
are finitely generated,
the obtained description is fully algorithmic and leads to a clean alternative proof 
for the solvability of the subgroup intersection problem $\SIP(\FTA)$; see \Cref{def: Howson group}.

\medskip

In~\Cref{sec: free-times-abelian groups} we introduce the family of free-times-abelian groups ($\FTA = \Fta$) together with some related terminology and notation. It turns out that this naive-looking family hides interesting features that translate into non-trivial problems; see \cite{carvalho_dynamics_FTA_2020,
DelgadoAlgorithmicProblemsFreeabelian2013, RoyFixedSubgroupsComputation2019,
RoyDegreesCompressionInertia2019, 
DelgadoRelativeOrderSpectrum2022,
delgado_configurations_2021}.

In \Cref{sec: enriched automata} we start by briefly surveying the classical Stallings theory for subgroups of the free group, to then introduce and study enriched automata (restricted to the free-times-abelian case). This leads to the classification \Cref{thm: enriched Stallings bijection}, which we use to derive first applications, such as the solvability of the membership problem and the computability of bases.

In \Cref{sec: Intersections FTA} we consider intersections of subgroups. After reviewing the classical pull-back technique for the free group, we develop the theory of enriched products to study intersections of subgroups in $\FTA$. The first important result is \Cref{thm: Stallings = Cayley} where we establish the relation between subgroup intersections and Cayley digraphs of abelian groups. Then, we focus on the algorithmic description of the intersection, which is summarized in \Cref{thm: enriched Stallings computable} and has two notable consequences:  a geometric proof of the solvability of the intersection problem $\SIP(\FTA)$, and the denial of any possible extension of the celebrated Hanna Neumann conjecture to any group containing $\Free[2] \times \ZZ$. Finally, we use a topological argument to extend the above ideas to non finitely generated intersections; this leads to \Cref{thm: enriched Stallings r.e.} providing a geometric description of arbitrary intersections within $\FTA$.

In \Cref{sec: index} we use these results to deduce a neat description of the cosets and index of a given finitely generated subgroup $\HH \leqslant \FTA$, which turns out to be transparently encoded in the enriched Stallings automata for $\HH$; see~\Cref{prop: index of HH}. A geometric solution for the finite index problem $\FIP(\FTA)$ and a description of a recursive transversal set easily follow.

Finally, in \Cref{sec: examples}, we provide some examples highlighting the most relevant aspects of our geometric construction.

\medskip
\goodbreak


We use lowercase boldface Latin font to denote abelian elements (${\vect{a},\vect{b}, \vect{c},\ldots}$), and
uppercase boldface Latin font to denote matrices with integer entries ($\matr{A},\matr{B},\matr{C},\ldots $). Capitalized calligraphic font is used to denote subgroups ($\HH,\mathcal{K},\mathcal{L},\ldots
$)
and subsets
($\mathcal{S},\mathcal{R}, \mathcal{T}, \ldots 
$) of $\FTA$, in contrast with  the corresponding objects in the factors, denoted by $H,K,L,\ldots$ and $R,S,T,\ldots$ respectively. Furthermore,
homomorphisms and matrices are assumed to act on the right;
that is, we denote by $(x)\varphi$ (or simply $x \varphi$) the image of the element $x$ by the homomorphism $\varphi$, and we denote by  $\varphi \psi$ the composition
\smash{$A \xto{\varphi} B \xto{\psi} C$}.
Accordingly, the image of the homomorphism associated to a matrix $\matr{A}$ is the \emph{row space} of $\matr{A}$, denoted by $\gen{\matr{A}}$.
Finally, we shall use the symbol $\infty$ to denote the \emph{countable} infinity.

\section{Free-times-abelian groups} \label{sec: free-times-abelian groups}

According to a very well-known classification theorem, any finitely generated abelian group is isomorphic to
\begin{equation*}
\label{eq: class fg abelian groups} \mathbb Z^{m'} \times (\mathbb Z/d_1\mathbb Z) \times \cdots \times (\mathbb Z /d_{m''}\mathbb Z) \,,
\end{equation*}
where $ m' , m'', d_1,\ldots ,d_{m''}$ are non-negative integers satisfying $2 \leq d_1 \,|\, d_2 \,|\, \cdots \,|\,d_{m''}$. We can think the elements of such a group as integral vectors of length $m=m'+m''$ whose $(m'+i)$-th coordinate works modulo $d_i$, for $i=1,\ldots ,m''$. For this reason, and assuming the list $d_1, \ldots ,d_{m''}$ of torsion orders fixed all along the  paper, we shall denote this abelian group simply as~$\AA$. We shall slightly abuse language and call an \defin{abelian basis} of $\AA$ any set of generators of the smallest possible cardinal, namely $m$.

We shall be interested in direct products of finitely generated free and abelian groups, namely groups of the form $\FTA = \Fta$.  The group $\FTA$ being non-abelian, it will be convenient to admit both additive and multiplicative notation for the elements in~$\AA \leqslant \FTA$; to this end, consider the standard presentation \begin{equation*}\label{eq: pres ATF}
\FTA
\,=\,
\Fta
\,=\,
\Pres{\!\!\begin{array}{c} x_1, \ldots ,x_n\\t_1, \ldots ,t_m
\end{array}} {\!\!
\begin{array}{ll}
t_i \, x_k = x_k \, t_i & \forall i\in [1,m],
\forall k\in [1,n] \!\!\\
t_{i} t_{j} = t_{j} t_{i} & \forall i,j \in [1,m] \\
(t_{m'+i})^{d_i}=1 & \forall i\in [1, m'']
\end{array} },
 \end{equation*}
and let us abbreviate their element normal forms $w(x)t_1^{a_1} t_2^{a_2} \cdots t_{m}^{a_m}$ just as $\fta{w}{a}$, where $\vect{a}=(a_1,\, \ldots ,\, a_m)\in \AA$, and $t$ is a formal symbol serving only as a pillar to hold the vector $\vect{a}$ up in the exponent. This way, the operation in $\FTA$ is given by $(\fta{u}{a})(\fta{v}{b})=\fta{uv}{a+b}$ in multiplicative notation, while the abelian part works additively, as usual, up in the exponent. In particular, the trivial element is $t^{(0,\ldots ,0)} \!=t^{\vect{0}}$, and  $t_i=t^{\vect{e_i}}$, where $\vect{e_i}=(0,\ldots ,1,\ldots ,0)$, $i=1,\ldots,m$, are the vectors in the canonical basis of $\AA$. We extend this notation to subsets $S \subseteq \AA \leqslant \FTA$, which are denoted by $\T^{S}$. For an element in normal form $\fta{w}{a}$, $w\in \Free_n$ is called its \emph{free part}, and the vector $\vect{a}\in \AA$ its \emph{abelian part}.

Note that the group $\FTA$ fits in the middle of the natural splitting short exact sequence
 \begin{equation}\label{eq: ses FTA}
\trivial \ \To \ \AA \ \xto{\,\iota\ } \ \FTA \ \xto{\prF} \ \Fn \ \To \ \trivial \ ,
 \end{equation}
where $\iota $ is the inclusion map, and $\prF$ is the projection to the free part $\fta{w}{a} \mapsto w$. The groups of this form are called \emph{free-times-abelian} and are the main object of study in the present paper.
It is straightforward to see that any subgroup $\HH \leqslant \FTA$ is again free-times-abelian; concretely, the restriction of \eqref{eq: ses FTA} to $\HH$ again gives a splitting (since $\HH\prF \leqslant \Fn$ is free) short exact sequence
 \begin{equation*}\label{eq: ses subgroups FTA}
\trivial \ \To \ \HH \cap \AA \ \xto{\,\iota\ } \ \HH
\ \xto{\pi_{\mid \HH}} \
\HH \prF \ \To \ \trivial \,,
 \end{equation*}
and it easily follows that
\begin{equation} \label{eq: subgroup decomposition}
\HH
\,=\, \HH \prF \sigma \times (\HH \cap \Zb^{m})
\,\isom \, \HH \prF \times (\HH \cap \Zb^{m}) \, ,
\end{equation}
where $\sigma$ is a (any) splitting of $\pi_{\mid \HH}$. Therefore, any subgroup $\HH \leqslant \FTA$ is iso\-mor\-phic to $\Free_{n'}\times A$, where $n'\in \NN\cup \{\infty\}$ and $A$ is a subgroup of $\AA$ (and so again finitely generated abelian, with a possibly different sequence of torsion orders). The claim below follows immediately and will become important in \Cref{sec: Intersections FTA}.

\begin{cor} \label{cor: H fg iff Hpi fg}
A subgroup $\HH \leqslant \Fta$ is finitely generated if and only if its projection $\HH \prF$ to the free part is finitely generated; otherwise, it is countably generated. \qed
\end{cor}

It is also obvious from \eqref{eq: subgroup decomposition} that
   \begin{equation*} \label{eq: rank subgroup}
   \rk (\HH)
   \,=\,
    \rk (\HH  \prF ) + \rk (\HH \cap \AA) \,.
   \end{equation*}
Taking respective basis for each factor
we reach our notion of basis for a subgroup of~$\FTA$.
\begin{defn}
A \defin{basis} of a subgroup $\HH \leqslant \Fta$ is a set of generators of $\HH$ of the form $\set{\fta{u_1}{a_1},\ldots ,\fta{u_p}{a_p};\fta{}{b_1},\ldots ,\fta{}{b_q}}$, where $\vect{a_1},\ldots ,\vect{a_p}\in \AA$, $\set{ u_1,\ldots ,u_p}$ is a free basis of~$H\prF$, and $\set{\vect{b_1},\ldots ,\vect{b_q}}$ is an abelian basis of $L_{\HH}=\HH\cap \AA$ (note that $\HH$ is finitely generated if and only if $p<\infty$). To avoid confusion, we reserve the word \defin{basis} for $\Fta$, in contrast with the terms \defin{abelian basis} and \defin{free basis} for the corresponding concepts in the abelian and free contexts, respectively.
\end{defn}

\begin{defn}
Given a subgroup $\HH \leqslant \FTA$ and an element $w \in \Fn$ we define the
\defin{(abelian) completion of $w$ in $\HH$} to be
$\cab{w}{\HH}=\set{\, \vect{a} \in \AA \st \fta{w}{a} \in \HH \,}$. We also say that $\vect{a}$ is \emph{a} \defin{completion} of $w$ in~$\HH$ if $\vect{a} \in \cab{w}{\HH}$.
\end{defn}

\begin{lem}~\label{lem: completion is a coset}
The completion $\cab{u}{\HH}$ is non-empty if and only if $u\in \HH \prF$ and, in this case, it is a coset of $L_{\HH} \coloneqq \HH \cap \AA$ in~$\AA$. In particular, if $\set{\fta{u_1}{a_1},\ldots, \fta{u_p}{a_p};\ta{b_1},\ldots,\ta{b_q}}$ is a basis for $\HH \leqslant \FTA$ and $w \in \Fn$, then
\begin{equation}\label{eq: 8}
\cab{w}{\HH} \,=\, \left\{\!
\begin{array}{ll}
\varnothing  & \text{if } w \notin \HH \prF \\
\bm{\omega} \matr{A} + L_{\HH} &  \text{if } w \in \HH \prF \ ,
\end{array}
\right.
\end{equation}
where $\matr{A}$ is the $p\times m$ matrix having $\vect{a_i}$ as $i$-th row, $L_{\HH}=\gen{\vect{b_1}, \ldots ,\vect{b_q}}\leqslant \AA$, and $\bm{\omega}=w \phi \rho$ is the abelianization of the expression of $w$ in base $\set{u_1, \ldots , u_p}$; that is, $\phi$ is the change of basis $\HH \prF \ni w \mapsto \omega$, where $w = \omega(u_1,\ldots,u_p)$, and $\rho$ is the abelianization $\Free[\set{u_1,\ldots,u_p}] \isom \Free[p] \onto \ZZ^{p}$; see~\Cref{fig: abelian completion diagram}.
\begin{figure}[H]
\centering
\begin{tikzcd}[row sep=0pt, column sep=25pt,ampersand replacement=\&]
   \Free[n] \, \geqslant \&[-30pt] \HH\prF \arrow[r,>->>,"\phi"] \& \Free[p]  \arrow[r,->>,"\rho"] \& \ZZ^p \arrow[r,->>,"\matr{A}"] \& \AA \arrow[r,"/L_{\HH}"] \& \AA\! / L_{\HH}\\
   \& w  \arrow[r,|->]\& \omega  \arrow[r,|->]\& \bm{\omega} \arrow[r,|->]
   \& \bm{\omega} \matr{A} \arrow[r,|->]
   \& \bm{\omega} \matr{A} + L_{\HH} \&[-30pt]  = \cab{w}{\HH} \, .
   \end{tikzcd}
   \caption{Completion diagram}
   \label{fig: abelian completion diagram}
\end{figure}
\end{lem}

An immediate consequence of the above discussion is the following useful equivalence.

\begin{rem}
Let $\fta{w}{a} \in \FTA$ and
$\HH \leqslant \FTA$
with basis
$\set{\fta{u_1}{a_1},\ldots, \fta{u_p}{a_p};\ta{b_1},\ldots,\ta{b_q}}$, then
\begin{equation*}
    \fta{w}{a} \in \HH
    \,\Biimp\,
    w \in \HH \prF
    \text{\, and \,}
    \vect{a} \in  w \phi \rho \matr{A} + L_{\HH}
\end{equation*}
\end{rem}

\begin{rem}
The natural extension of Lemma~\ref{lem: completion is a coset} works as well for non finitely generated subgroups $\HH \leqslant \Fta$. In this case, $p=\infty$, a basis for $\HH$ looks like $\set{\fta{u_1}{a_1},\ldots ;\ta{b_1},\ldots,\ta{b_q}}$, and~\cref{eq: 8} is true as written, understanding that $\matr{A}$ is an integral matrix with countably many rows and $m$ columns, and that $\ZZ^\infty$ means $\bigoplus_{i=1}^{\infty} \ZZ$. Note that, then, $\bm{\omega}$ is a row vector with countably many coordinates, all but finitely many of them being $0$; so, the product $\bm{\omega} \matr{A}$ still makes sense with the usual meaning.
\end{rem}

\section{Enriched automata} \label{sec: enriched automata}

In this section we briefly survey the basics on the classical Stallings automata, to then develop our enriched theory (restricted to free-times-abelian groups $\Fta$,
see \cite{DelgadoExtensionsFreeGroups2017,delgado_Stallings_free_ext} for a more general and detailed account, including the case of semidirect products).

This geometric approach dates back to the 1980's, with the ideas of Serre, Stallings and others (see \cites{SerreTrees1980, StallingsTopologyFiniteGraphs1983}) interpreting the subgroups of the free group $\Free[n] =\pres{X}{-}$ as covering spaces of the bouquet of $n$ circles. This topological viewpoint was later reformulated in a more combinatorial way in terms of pointed $X$-automata --- that is, digraphs labeled by letters in $X$ with a distinguished (initial and terminal) vertex --- and can be summarized in \Cref{thm: Stallings bijection};
see \cite{KapovichStallingsFoldingsSubgroups2002, BartholdiRationalSubsetsGroups2010} for details and proofs. The precise notion of automaton used in this context is stated below.

The \defin{involutive closure} of a set $X$ (usually understood as an alphabet) is the disjoint union $X^{\pm} \coloneqq X \sqcup X^{-1}$, where $X^{-1} \coloneqq \set{x^{-1} \st x\in X}$ is the set of \defin{formal inverses} of $X$.

\begin{defn} \label{def: involutive automaton}
        Let $X$ be a set.
        By an \defin{(involutive and pointed) $X$-automaton}
        $\Ati$
        we mean a
        $X^{\pm}$-labeled digraph
      such that
      for every arc
        $\smash{\edgi \equiv \verti \xarc{\,x\ } \vertii}$ 
        (reading $x \in X^{\pm}$) there exists a unique (\defin{inverse}) arc
        $\smash{\edgi^{-1} \equiv \verti \xcra{\,x^{-1}} \vertii}$
        (reading $x^{-1})$, with
       a distinguished vertex $\bp$ called the \defin{basepoint} of~$\Ati$ (which acts as the unique initial and accepting vertex for $\Ati$).     
\end{defn}


If $\edgi \equiv \verti \xarc{x\,} \vertii$ is an arc in $\Ati$, then we say
that $\verti$ and $\vertii$ are respectively the
\defin{initial vertex} or \defin{origin}  of $\edgi$ (denoted by $\init \edgi$), and the \defin{terminal vertex} or \defin{end} of $\edgi$ (denoted by $\term \edgi$); and that $x$ is the label of $\edgi$, denoted by $\lab_{X}(\edgi)$.
We also say that the vertices $\verti,\vertii$ are \defin{adjacent}, and that the arc
$\edgi$ is \defin{incident to}
both $\verti$ and $\vertii$.
The sets of vertices and arcs of $\Ati$ are denoted by $\Verts \Ati$ and $\Edgs \Ati$ respectively.
An involutive $X$-automaton
is said to be \defin{saturated}
(or \defin{complete})
 if every vertex is the origin of an $x$-arc, for every~$x\in X^{\pm}$.
 
A \defin{walk}
in an automaton $\Ati$
is a finite alternating sequence
   $\walki = \verti_0 \edgi_1 \verti_1 \cdots \edgi_{n} \verti_{n}$ such that
   $\init \edgi_i = \verti_{i-1}$ and 
   $\term \edgi_i = \verti_i$, for $i \in [1,n]$.
   If $\verti_0 = \verti_n$ we say that $\walki$ is a (closed) \defin{$\verti_0$-walk}.
   The \defin{length of a walk} is the number of arcs in the sequence. Walks of length $0$ correspond precisely to the vertices in~$\Ati$.
   A walk is said to exhibit \defin{backtracking} if it has two consecutive arcs which are inverses of each other, and is called \defin{reduced} otherwise.

The \defin{label}
(\resp \defin{free label})
of a walk $\walki$ is the element in $(X^{\pm})^*$
(\resp in $\Free[X]$)
given
(\resp represented)
by the sequence of labels in the arcs of~$\walki$,
assumed to be the empty string $\emptyword$
(\resp the trivial element $\trivial$)
if the walk
consists of just a vertex.
It is easy to see that the set of free labels of $\bp$-walks in an involutive $X$-automaton $\Ati$ is a subgroup of $\Free[X]$. It is called the \defin{subgroup recognized} by $\Ati$, denoted by~$\gen{\Ati}$.

We denote by $\Edgs^{+}(\Ati)$ the subset of arcs in $\Ati$ labeled by elements in $X$ (which we call the \defin{positive arcs} of $\Ati$).
Note that we can represent involutive automata using only the positive arcs of $\Ati$ 
(this is called the \defin{positive part} of $\Ati$),
with the convention that every 
$x$-arc $\edgi$ can also be crossed backwards, reading $x^{-1}$ (corresponding to the hidden
inverse arc $\edgi^{-1}$).
Unless stated otherwise, the automata appearing throughout the paper will be assumed to be pointed and involutive. We will refer to them simply as automata.
Note that if, in an involutive automaton $\Ati$, we identify mutually inverse arcs
and ignore the labeling and basepoint, then we obtain an undirected multigraph, which we call the \defin{underlying (undirected) graph}
of~$\Ati$.

If a graph $\Gri$ can be obtained by identifying a vertex of some graph $\Grii$ with a vertex of some disjoint non-trivial tree $\Treei$,
then we say that $\Treei$ is a \defin{hanging tree} of $\Gri$.
A hanging tree is \defin{maximal} if it is not contained in any other hanging tree. Both notions extend naturally to involutive automata via the corresponding  underlying graphs.

\begin{defn}
An $X$-automaton is said to be \defin{deterministic} if
no two arcs with the same label depart from
(or arrive at)
the same vertex;
and
\defin{core} if every vertex appears in some reduced \mbox{$\bp$-walk}. Note that being core is equivalent to being connected and having no hanging trees not containing the basepoint.
The \defin{core} of an automaton $\Ati$,
denoted by $\core (\Ati)$,
is the maximum core subautomaton of $\Ati$,
\ie the automaton obtained after taking the basepoint component of $\Ati$ and removing from it all the hanging trees not containing the basepoint.
Note that $\gen{\core (\Ati)} = \gen{\Ati}$.
Finally, an $X$-automaton is said to be \defin{reduced} if it is both deterministic and core.
\end{defn}

Important examples of pointed involutive automata
are Schreier and Stallings automata, which we define below.
\begin{defn}
Let $\Free$ be a free group with basis $\basis$
and let $H$ be a subgroup of $\Free$. The 
\defin{(right) Schreier automaton} of $H$ \wrt $\basis$,
denoted by $\schreier{H,\basis}$,
is the automaton with set of vertices
$H \backslash \Free$
(the set of right cosets of $H$);
an arc
$Hu \xarc{v\,} Huv$
(from $Hu$ to $Huv$ labeled by~$v$)
for every coset $Hu \in H \backslash \Free$
and every element $v \in \basis$;
and the coset $H$ as basepoint.
\end{defn}
Note that Schreier automata are always connected, deterministic, and saturated, but not necessarily core. The core of $\schreier{H,\basis}$ is a reduced (involutive and pointed) \mbox{$\basis$-automaton}, called the \defin{Stallings automaton} of~$H$ (\wrt $\basis$) and denoted by $\Stallings{H,\basis}$; that is,
$\stallings{H,\basis} = \core(\schreier{H,\basis})$.
Clearly, $\gen{\schreier{H,\basis}}=\gen{\Stallings{H,\basis}}=H$.
Note that both Schreier and Stallings automata
depend on the free basis chosen for the ambient group, \emph{and hence on the ambient group itself}. (We alert the reader that, throughout the paper, Stallings automata relative to different ambient groups and bases shall be considered for the same subgroup.)


\begin{thm}[\citenr{StallingsTopologyFiniteGraphs1983}]\label{thm: Stallings bijection}
Let $\Free[X]$ be a free group with basis $X$.
Then, the map
 \begin{equation}\label{eq: Stallings bijection}
\begin{array}{rcl}
\operatorname{St}\colon \set{\,\text{subgroups of } \Free[X] \,} & \leftrightarrow & \set{\,\text{(isomorphic classes of) reduced $X$-automata}\,} \\
H & \mapsto & \stallings{H,X} \coloneqq \core (\schreier{H,X})\\
\gen{\Ati} & \mapsfrom & \Ati
\end{array}
 \end{equation}
is a bijection. Furthermore, finitely generated subgroups correspond precisely to finite automata and, in this case, the bijection is algorithmic.
\end{thm}

To compute $\stallings{H,X}$ (given a \emph{finite} set of generators $S$ for $H$) we start by building the so-called \defin{flower automaton}  $\flower{S}$ of $S$, which is
obtained after identifying the basepoints of
the
(involutive)
petals spelling the generators in $S$, which we can assume to be reduced words.
Note that, by construction, $\flower{S}$ is core and recognizes $H$, but may fail to be deterministic at the basepoint.
To fix this, one can successively identify the possible arcs breaking determinism. It is clear that these identifications, called \defin{foldings}, do not change the recognized subgroup. Of course, a folding can produce new nondeterministic situations to be fixed, but since the number of arcs in the graph is finite, and decreases with each folding, the process finishes after a finite number of steps, producing as a result a deterministic $X$-automaton recognizing $H$. Moreover, since the folding process can only produce hanging trees containing the basepoint, the final object is still core, and hence a reduced $X$-automaton recognizing $H$. \Cref{thm: Stallings bijection} states that this resulting automaton must be precisely $\Stallings{H,X}$. Furthermore, the bijectivity of \eqref{eq: Stallings bijection} implies that the result of the folding process
depends neither on the order in which the foldings are performed, nor on the starting (finite) generating set taken for $H$,
but only on the subgroup~$H\leqslant \Free[X]$ itself.

For the opposite direction, suppose we are given a finite reduced $X$-automaton $\Ati$. Consider a spanning tree \!$\Treei$ of $\Ati$ and
denote by
\smash{
$\verti \xleadsto{_{\Treei}} \vertii$}
the unique reduced walk from a vertex $\verti$ to a vertex $\vertii$ using only arcs in $\Treei$;
and by
$\Tpetal{\edgi}$
the $\bp$-walk
\smash{$\bp \xleadsto{_{\scriptscriptstyle{T}}} \! \bullet \! \arc{\edgi} \!\bullet\! \xleadsto{_{\scriptscriptstyle{T}}} \bp$}, where ${\edgi\in \Edgs\Ati \setmin \Edgs \Treei}$. It is not difficult to see that the set $\basis_{\Treei} \coloneqq \set{ \lab_X(\Tpetal{\edgi}) \st \edgi\in \Edgs^{+} \Ati \setmin \Edgs \Treei \, }$ constitutes a free basis of the subgroup $\gen{\Ati} \leqslant \Free[X]$. We say that $\basis_{\Treei}$ is the \defin{(positive) $\Treei$-basis} of $\gen{\Ati}$, that the $\Tpetal{\edgi}$'s are the \defin{(positive) $\Treei$-petals},
and that the $\edgi$'s are the (\defin{positive cyclomatic}) \defin{$\Treei$-arcs} of~$\Ati$.


Since the Stallings automaton of any finitely generated subgroup $H\leqslant \Fn$ is computable, we can immediately  compute a basis for $H$ as described above, and decide membership for $H$ simply by checking whether the candidate reduced word $w\in \Fn$ labels a $\bp$-walk in~$\Stallings{H,X}$. Other well-known algorithmic applications of  \Cref{thm: Stallings bijection} include
the study of intersections (see~\Cref{sec: Intersections FTA}), and the description of finite index subgroups (see~\Cref{sec: index}). 
Also, the classical Nielsen-Schreier Theorem follows immediately: any subgroup $H \leqslant \Fn$ is the fundamental group of the underlying graph of $\stallings{H,X}$ and hence it is free.

\medskip

In \cite{DelgadoExtensionsFreeGroups2017} we developed a broader generalization of Stallings' techniques oriented towards extensions of the form $\Fn \ltimes \ZZ^m$, not yet available in published form. Below, we present this theory restricted to the case of free-times-abelian groups.
Our fundamental object is an extension of the $X$-automata used in the free case: we shall also admit abelian labels at the end and origin of every arc, and a subgroup of~$\AA$ labeling the basepoint of the automata. The precise definition follows.

\begin{defn} \label{def: enriched automaton}
A \defin{$\AA$\!-enriched $X$-automaton} (\defin{enriched automaton} for short) is a pointed involutive $(\AA \! \times \Geni \times \AA)$-automaton, with a subgroup of $\AA$ attached to the basepoint. In more detail, an enriched automaton $\Eti$ consists of:
\begin{enumerate}[ind]
\item an involutive pointed digraph $\dGri = (\Verts,\Edgs, \init, \term, \bp)$ (the \defin{underlying digraph} of $\Eti$);
\item an involtive arc-labeling $\vec{\lab} = (\lab_1,\lab_{X},\lab_{2}) \colon \Edgi \to \AA \! \times \Geni \times \AA$ (the \defin{enriched labeling} of $\Eti$); \ie
for every arc
$\smash{\edgi \equiv \verti \xarc{\ } \vertii}$ labeled by $(\vect{a},x, \vect{b})$ 
there exists a unique (\defin{inverse}) arc
$\smash{\edgi^{-1} \equiv \verti \xcra{\ } \vertii}$
labeled by $(-\vect{b},x^{-1}, -\vect{a})$.
\item a subgroup $L_{\Eti} \leqslant \AA $ attached to the basepoint $\bp$ of $\Eti$ (the \defin{basepoint subgroup} of~$\Eti$).
\end{enumerate}

The \defin{body} of an enriched automaton $\Eti$, denoted by $\body{\Eti}$, is the result of removing from $\Eti$ the basepoint subgroup;
whereas
the \defin{skeleton} of $\Eti$, denoted by $\sk (\Eti)$, is the result of removing all the abelian information
(\ie the basepoint subgroup and all the abelian labels)
from~$\Eti$.
 Note that $\sk (\Eti) = \sk (\body{\Eti})$ is a standard $X$-automaton.
An enriched $X$-automaton $\Ati$ is said to be \defin{deterministic} (\resp \defin{connected}, \defin{core}, \defin{reduced}) if its skeleton $\sk(\Ati)$ is so, and we define the core of an enriched automaton accordingly.
\end{defn}
If an arc
$\edgi \equiv \verti \xarc{\ } \vertii$ is labeled by $(\vect{a}, x,\vect{b})$ then we write 
\begin{tikzpicture}[baseline=-0.5ex,shorten >=1pt, node distance=.3cm and 1.5cm, on grid,>=stealth']
\node[] (p) {$\verti$};
\node[] (q) [right = 1.75 of p] {$\vertii$};
\node[] (e) [left = 0.5 of p]{$\edgi \equiv$};
\path[->]
        (p) edge[]
            node[pos=0.5,below=-.2mm] {$\displaystyle{x}$}
            node[pos=0.1,above=-.2mm] {$\displaystyle{\mathbf{a}}$}
            node[pos=0.8,above=-.2mm] {$\displaystyle{\mathbf{b}}$}
            (q);
\end{tikzpicture}
(with the first and second abelian labels at the beginning and end of the enriched arc, and the free label in the middle).
As in the free case, the idea is that the labeling (of the arcs) in an enriched automaton $\Eti$ extends to a $\FTA$-labeling on the walks (sequences of successively adjacent arcs) in $\Eti$. For enriched automata the rules are the following:
\begin{enumerate}
    \item Every arc  
\smash{
\begin{tikzpicture}[baseline=-0.5ex,shorten >=1pt, node distance=.3cm and 1.5cm, on grid,>=stealth']
\node[state] (p) {};
\node[state] (q) [right = 1.5 of p] {};
\path[->]
        (p) edge[]
            node[pos=0.5,below=-.2mm] {$\displaystyle{x_j}$}
            node[pos=0.1,above=-.2mm] {$\displaystyle{\mathbf{a}}$}
            node[pos=0.8,above=-.2mm] {$\displaystyle{\mathbf{b}}$}
            (q);
\end{tikzpicture}}
\,in $\Eti$ is meant to be read $\mathrm{t}^{-\vect{a}} \, \geni_{j} \, \mathrm{t}^{\vect{b}} =\geni_{j} \, \mathrm{t}^{\vect{b}-\vect{a}}$ when crossed forward (from left to right), and $\mathrm{t}^{-\vect{b}} \, \geni_{j}^{-1} \, \mathrm{t}^{\vect{a}} = \geni_{j}^{-1} \, \mathrm{t}^{\vect{a}-\vect{b}} =(\geni_{j} \, \mathrm{t}^{\vect{b}-\vect{a}})^{-1}$ when crossed backwards (from right to left).
\item Successive arcs in a walk read the product (in $\FTA$) of the labels of the arcs.
\item Elements from $L_{\Eti}$ are thought of as labeling  ``infinitesimal'' commuting loops at $\bp$, that is, when at $\bp$ one can freely pick an element from $L_{\Eti}\leqslant \AA\leqslant \FTA$ as a label.
\end{enumerate}


More precisely, the \emph{enriched label} of a non-trivial walk $\walki =\edgi_1^{\epsilon_1}\cdots \edgi_k^{\epsilon_k}$ in $\Ati$, $k\geqslant 1$, is $\llab(\walki)=\llab(\edgi_1)^{\epsilon_1}\cdots \llab(\edgi_k)^{\epsilon_k}$, where $\llab(\edgi_i)=\mathrm{t}^{-\lab_1(\edgi_i)}\, \lab_{X}(\edgi_i) \, \mathrm{t}^{\lab_{2}(\edgi_i)} \in \FTA$; note that the label of $\walki$ as a walk in the skeleton is just $\lab_{X}(\walki) = \lab_{X}(\edgi_1)^{\epsilon_1}\cdots \lab_{X}(\edgi_k)^{\epsilon_k} \in \Fn$.
As a convention, we admit any element in $L_{\Eti}$ as a possible label of the trivial $\bp$-walk.

Recall that a walk beginning and ending at the basepoint is called a $\bp$-walk. An element (in $\FTA$) labeling a $\bp$-walk in an enriched automaton $\Eti$ is said to be \defin{recognized by} $\Eti$; for example, every $\vect{l}\in L_{\Eti}$ is so. It is straightforward to check that the set of all the elements recognized by an enriched automaton $\Eti$ is a subgroup of $\FTA$: it is called the \defin{subgroup recognized by $\Eti$}, and denoted by $\gen{\Eti}$. Note that $\gen{\sk(\Eti)} = (\gen{\Ati}) \prF$, and $\gen{\Eti} = \gen{\core (\Eti)}$.

It is clear that every subgroup in $\FTA$ is recognized by some enriched automata. This is obvious for subgroups inside $\AA \leqslant \FTA$ (which can be set as basepoint subgroups); on the other hand, given any element $\fta{u}{a} \in \FTA$ with $u\neq 1$, we can always consider the \defin{petal automaton} $\flower{\fta{u}{a}}$; that is, the following directed $\bp$-walk:
\begin{figure}[H]
\centering
  \begin{tikzpicture}[shorten >=1pt, node distance=.3cm and 1.5cm, on grid,>=stealth']
   \node[state,accepting] (0) {};
   \node[state] (1) [right = of 0] {};
   \node[state] (2) [right = of 1] {};
   \node[] (dots) [right = 0.5 of 2] {$\cdots$};
   \node[state] (4) [right = 0.5 of dots] {};
   \node[state, accepting] (5) [right = of 4] {};
   \node[](i) [right = 0.5 of 0] {};
   \node[](f) [left = 0.5 of 5] {};

   \path[->]
        (0) edge[]
            node[pos=0.5,below=-.2mm] {$\displaystyle{x_{i_{\textsf{1}}}}$}
            node[pos=0.1,above=-.2mm] {$\displaystyle{\mathbf{0}}$}
            node[pos=0.8,above=-.2mm] {$\displaystyle{\mathbf{0}}$}
            (1);
   \path[->]
        (1) edge[]
            node[pos=0.5,below=-.2mm] {$\displaystyle{x_{i_{\textsf{2}}}}$}
            node[pos=0.1,above=-.2mm] {$\displaystyle{\mathbf{0}}$}
            node[pos=0.8,above=-.2mm] {$\displaystyle{\mathbf{0}}$}
            (2);

   \path[->]
        (4) edge[]
            node[pos=0.5,below=-.2mm] {$\displaystyle{x_{i_k}}$}
            node[pos=0.1,above=-.2mm] {$\displaystyle{\mathbf{0}}$}
            node[pos=0.8,above=-.2mm] {$\displaystyle{\mathbf{a}}$}
            (5);
\end{tikzpicture}
\caption{A petal automaton recognizing $x_{i_1} x_{i_2} \cdots x_{i_k} \ta{a} = u\ta{a}$}
\label{fig: enriched petal automaton}
\end{figure}
Note that the label of this cycle is $\fta{u}{a}$ and hence  $\gen{\flower{\fta{u}{a}}}=\gen{\fta{u}{a}}$. Then, given a finite
subset $\mathcal{S}=\set{\fta{u_1}{a_1}, \ldots, \fta{u_p}{a_p},\ta{b_1},\ldots,\ta{b_q}} \subseteq \FTA$, with $u_1, \ldots ,u_p\neq 1$, we define the \defin{flower automaton} $\flower{\mathcal{S}}$ as the result of identifying the basepoints of the petals of the first $p$ elements in~$\mathcal{S}$, and declaring the basepoint subgroup to be $L_{\Eti}=\gen{\vect{b_1},\ldots,\vect{b_q}}$; see~\cref{fig:  flower automaton}.
\vspace{-20pt}
\begin{figure}[H]
\centering
\begin{tikzpicture}[shorten >=1pt, node distance=2.5cm and 2.5cm, on grid,auto,>=stealth',
decoration={snake, segment length=2mm, amplitude=0.5mm,post length=1.5mm}]
   \node[state,accepting] (1) {};
   \node (L) [below right= 0.4 and 0.75 of 1] {$\scriptstyle{L_{\Eti} = \gen{\vect{b_1},\ldots,\vect{b_q}}}$};
   \path[->,thick]
        (1) edge[loop,out=120,in=160,looseness=8,min distance=25mm,snake it]
            node[left=0.1] {$u_1$}
            node[below,pos=0.95] {$\vect{a_1}$}
            (1);
            (1);
   \path[->,thick]
        (1) edge[loop,out=60,in=20,looseness=8,min distance=25mm,snake it]
            node[right=0.1] {$u_p$}
            node[below,pos=0.95] {$\vect{a_p}$}
            (1);
\foreach \n [count=\count from 0] in {1,...,3}{
       \node[dot] (1\n) at ($(1)+(73+\count*15:0.75cm)$) {};}
\end{tikzpicture}
\vspace{-10pt}
\caption{The flower automaton $\flower{\mathcal{S}}$}
\label{fig:  flower automaton}
\end{figure}
\vspace{-10pt}
Clearly, one can extend the definition of flower automata to infinite subsets in the obvious way, and, in any case, $\gen{\flower{\mathcal{S}}} = \gen{\mathcal{S}}$, where the eventual purely abelian elements in $\mathcal{S}$ generate the basepoint subgroup $L_{\Eti}$. It is important to realize that although $L_{\Eti}\leqslant \HH \cap \AA$, the opposite inclusion may not be true, due to possible non-trivial relations among the free parts $u_1, \ldots ,u_p$.

Of course, a given subgroup $\HH\leqslant \FTA$ can be recognized by (infinitely) many enriched automata.
Namely,
\begin{enumerate*}[ind]
\item
the skeleton of the flower automaton defined above depends on (the free parts of) the chosen set of generators $\mathcal{S}$ for $\HH$; and there is also a lot of freedom in the distribution of the abelian labeling since:
\item
for any petal, we could alternatively have put the $\vect{a}$ label at the end of any of the other arcs in the walk (among infinitely many other possible configurations reading the same element $\fta{u}{a}$);
and
\item
every abelian label in $\Eti$ works modulo the basepoint subgroup $L_{\Eti}$.
\end{enumerate*}
So, the map $\Ati\mapsto \gen{\Ati}$ from the set of enriched automata to the set of subgroups of $\FTA$ is onto but very far from injective. To make it bijective we have to distinguish one and only one geometric object recognizing each subgroup.

\begin{defn}\label{def: normalized automaton}
Let $\Eti 
$ be an enriched $X$-automaton, and let $\Treei$ be a spanning tree of~$\Eti$.
We say that $\Ati$ is $\Treei$-\defin{normalized} if it is reduced, and the abelian labels of $\Ati$ are concentrated at the ends of the arcs outside $\Treei$ (\ie $\lab_1(\edgi)=\vect{0}$  for every $\edgi\in \Edgi\Ati$, and $\lab_{2}(\edgi)=\vect{0}$ for every $\edgi \in \Edgi\Treei$). It is easy to see that, if $\Eti
$ is a $\Treei$-normalized automaton recognizing $\HH$, then $\sk (\Eti) = \stallings{\HH \prF,X}$, and $L_{\Eti}=\HH \cap \AA$; see~\Cref{prop: enriched Stallings parts}.

\end{defn}

It is not difficult to see that, after taking the quotient modulo the basepoint subgroup $ L_{\Eti}$ (denoted by ``$\operatorname{mod}$ $\bp$''), we finally reach the desired unicity: for any given subgroup $\HH\leqslant \FTA$, and any given spanning tree $\Treei$ of $\Stallings{\HH \prF, X}$, every two \mbox{$\Treei$-normalized} enriched automata 
recognizing $\HH$ are equal modulo $L_{\Eti}$.
This uniquely determined object is called the $\Treei$-\defin{Stallings automaton} for $\HH$, denoted by $\St_{\Treei}(\HH,X)$. When the spanning tree $\Treei$ is clear from the context we will usually omit any reference to it
and write $\St(\HH,X)$.
Also, since unicity is usually not necessary for computational purposes, we will often
abuse terminology
and call any normalized automaton recognizing $\HH$ a ``\emph{Stallings automaton for $\HH$}''.

Finally, in order to obtain the desired bijection, we need a uniform way of distinguishing spanning trees in all the enriched automata. This can be done by fixing a total order $\preccurlyeq$ in the set $X\cup X^{-1}$: for any given $\Ati$, declare that~$\bp$ is in $\Treeo$ and then, recursively, add to $\Treeo$ the edge (together with its other incident vertex) with smallest possible label incident to the oldest vertex present in $\Treeo$ at that moment and not closing a path. This determines (even in the infinite case) a spanning tree in $\Eti$  denoted by $\Treeo(\Ati)$; see~\cite{DelgadoExtensionsFreeGroups2017,delgado_Stallings_ABF,delgado_Stallings_free_ext} for details. We say that $\Eti$ is $\preccurlyeq$-\defin{normalized} if it is $\Treeo(\Eti)$-normalized, and we write $\St_{\preccurlyeq}(\HH, X) \coloneqq \St_{\Treeo}(\HH, X)$.

The main result in this section is the following bijection between subgroups of $\FTA$ and (uniformly chosen) enriched Stallings automata, which are furthermore computable in the finitely generated case.

\begin{thm}
\label{thm: enriched Stallings bijection}
Let $\Free[X]$ be a free group with finite basis $X$, let~$\AA$ be a finitely generated abelian group, and let $\preccurlyeq$ be a total order on $X^{\pm}$. Then, the map
 \begin{equation}\label{eq: enriched Stallings bijection}
\begin{array}{rcl}
\operatorname{St}_{\preccurlyeq}\colon \big\{\,\text{subgroups of } \Free[X] \times \AA  \big\} & \leftrightarrow &
\Big\{
\begin{array}{l}
\text{(isomorphic classes of) }{\preccurlyeq}
\text{-normalized}\\
\text{$\AA$\!\!-enriched $X$-automata mod \bp}
\end{array}
\Big\} \\
\HH \ & \mapsto & \  \St_{\preccurlyeq}(\HH,X) \\
\gen{\Ati} \ & \mapsfrom &\  \Ati
\end{array}
 \end{equation}
is a bijection. Furthermore, finitely generated subgroups correspond precisely to finite automata and, in this case, the bijection is algorithmic.
\end{thm}

Let us focus on the algorithmic behavior of bijection \eqref{eq: enriched Stallings bijection}. Given a finite family of generators for a subgroup $\HH \leqslant \FTA$, we can algorithmically obtain a Stallings automaton recognizing $\HH$ by constructing the corresponding (enriched) flower automaton and appropriately adapting the folding process to the enriched scenario. To this end, we introduce two new ``abelian transformations'' intended to move the abelian mass around the automaton without changing the recognized subgroup.

\begin{defn}  \label{def: abelian transformations}
A \defin{vertex transformation} consists in adding a vector $\vect{c} \in \AA$ to every abelian label in the neighborhood of a vertex $\verti$:

\begin{figure}[H]
\centering
  \begin{tikzpicture}[shorten >=1pt, node distance=2cm and 2cm, on grid,auto,>=stealth']

   \node[state] (0) {};
   \node[] (1) [above = 1.5cm of 0] {};
   \node[] (2) [left = 2cm of 0] {};
   \node[] (3) [below right = 1.5cm of 0] {};

   \node[] (i) [] [right = 40pt of 0] {};
   \node[] (f) [] [right = 33pt of i] {};

   \node[state] (0') [right = 2.5cm of f]{};
   \node[] (1') [above = 1.5cm of 0'] {};
   \node[] (2') [left = 2 cm of 0'] {};
   \node[] (3') [below right = 1.5cm of 0'] {};

   \path[->]
        (0) edge[]
            node[pos=0.55,left=-.5mm] {$x_{i_1}$}
            node[pos=0.15,right=-.2mm] {$\vect{a_1} $}
            (1);
   \path[->]
        (2) edge[]
            node[pos=0.35,below=-.2mm] {$x_{i_2}$}
            node[pos=0.8,above=-.2mm] {$\vect{a_2} $}
            (0);
   \path[->]
        (3) edge[]
            node[pos=0.5,below] {$x_{i_3}$}
            node[pos=0.85,right= .5mm] {$\vect{a_3}$}
            (0);

   \path[->,thin,>=to]
        (i) edge[bend left]
         (f);

    \path[->]
        (0') edge[]
            node[pos=0.55,left=-.5mm] {$x_{i_1}$}
            node[pos=0.15,right] {$ \vect{a_1}  \!+\! \vect{c}$ }
            (1');
   \path[->]
        (2') edge[]
            node[pos=0.3,below=-.2mm] {$x_{i_2}$}
            node[pos=0.7,above=-.2mm] {$\vect{a_2} \!+\! \vect{c}$}
            (0');
   \path[->]
        (3') edge[]
            node[pos=0.5,below] {$x_{i_3}$}
            node[pos=0.85,right= .5mm] {$\vect{a_3} \!+\! \vect{c} $}
            (0');
\end{tikzpicture}
\caption{A vertex transformation}
\end{figure}

An \defin{arc transformation} consists on adding a vector $\vect{c} \in \AA$ to both the initial and final abelian labels of an arc:

\begin{figure}[H]
\centering
  \begin{tikzpicture}[shorten >=1pt, node distance=.3cm and 2.8cm, on grid,auto,>=stealth']
   \node[state] (0) {};
   \node[state] (1) [right = of 0] {};

   \node[] (i) [] [right = 15pt of 1] {};
   \node[] (f) [] [right = 33pt of i] {};
   \node[state] (0') [right = 15pt of f] {};
   \node[state] (1') [right = of 0'] {};

   \path[->]
        (0) edge[]
            node[pos=0.5,below=-.1mm] {$x_{i}$}
            node[pos=0.07,above=-.1mm] {$\vect{a}$}
            node[pos=0.88,above=-.1mm] {$\vect{b}$}
            (1);

   \path[->,thin,>=to]
        (i) edge[bend left]
         (f);

    \path[->]
        (0') edge[]
            node[pos=0.47,below=-.1mm] {$x_{i}$}
            node[pos=0.15,above=-.1mm] {$\vect{a}\!+\! \vect{c}$}
            node[pos=0.75,above=-.3mm] {$\vect{b} \!+\!  \vect{c}$}
            (1');
\end{tikzpicture}
 \caption{An arc transformation}
 \end{figure}
It is obvious that these two abelian transformations do not affect the skeleton of the automaton, and it is straightforward to check that they do not affect the recognized subgroup either.
Note that a vertex transformation at the basepoint (say by a vector $\vect{c}$) corresponds to a conjugation by $\mathrm{t}^{\vect{c}}$, which in our case belongs to the center of $\Fta$.
\end{defn}

We claim that these two abelian transformations suffice to convert any folding situation in $\sk (\Eti)$ into a folding situation in $\Eti$: suppose that $\edgi$ and $\edgii$ are two arcs in $\Eti$ with the same free label $\lab_X(\edgi)=\lab_X(\edgii)$
departing from the same vertex, say $\verti = \iota \edgi =\iota \edgii$. Distinguish two cases: the \defin{open} case, when they are non-parallel (\ie $\tau \edgi \neq\tau \edgii$), and the \defin{closed} case when they are parallel (\ie $\tau \edgi =\tau \edgii$).

In the open case, in order to fold $\edgi$ and $\edgii$, we have to make sure that both arcs have the same abelian labels: performing an appropriate arc transformation to $\edgii$ we can get $\lab_1(\edgi)=\lab_1(\edgii)$; and then, after an appropriate vertex transformation at $\tau \edgii$ (and using the fact $\tau \edgii \neq\tau \edgi$), we can further obtain $\lab_2(\edgi)=\lab_2(\edgii)$.
After this preparation, all the labels in $\edgi$ and $\edgii$ coincide, and we can effectively perform the folding in $\Ati$.

Note that the above procedure does not work in the closed situation because the vertex transformation at $\tau \edgii$ also affects the label $\lab_2(\edgi)$ we want to match. In this case, instead, we just fully remove $\edgii$ and update the basepoint subgroup from $L_{\Eti}$ to $L_{\Eti}+\gen{-\lab_2(\edgi)+\lab_1(\edgi)-\lab_1(\edgii)+\lab_2(\edgii)}$ in order to take into account the purely abelian contribution of the closed walk around the folded cycle.

\begin{figure}[H]
\centering
  \begin{tikzpicture}[shorten >=1pt, node distance=1cm and 2.5cm, on grid,auto,baseline,>=stealth']
   \node[state] (1)  {};
   \node[state] (2) [right = of 1] {};
   \node[state,accepting] (bp) [below right = 1.25  and 1.25 of 1] {};
   \node[] (H) [below = 0.3 of bp] {$\scriptstyle{L_{\Eti}}$};

   \path[->]
        (1) edge[bend left=40]
            node[below] {$\geni_i$}
            node[pos=0.05,above = 0.02] {\rotatebox[origin=c]{30}{$\vect{a}$}}
            node[pos=0.92,above = -0.02] {\rotatebox[origin=c]{-30}{$\vect{b}$}}
            (2);
   \path[->]
        (1) edge[bend right=40]
            node[above= -0.02] {$\geni_i$}
            node[pos=0.05,below = 0.02] {\rotatebox[origin=c]{30}{$\vect{c}$}}
            node[pos=0.92,below = -0.02] {\rotatebox[origin=c]{30}{$\vect{d}$}}
            (2);
   \path[->]
        (2) edge[snake it,bend left=41]
        (bp);

   \node[] (i) [right = 0.75 of 2]{};
   \node[] (f) [right = 1.15 of i] {};

    \path[->,thin,>=to]
        (i) edge[bend left]
         (f);

    \node[state] (1') [right = 0.75 of f] {};
   \node[state] (2') [right = of 1'] {};
   \node[state,accepting] (bp') [below right = 1.25  and 1.25 of 1'] {};
   \node[] (H) [below right = 0.35 and 0.95 of bp'] {$\scriptstyle{L_{\Eti} \,+\, \gen{(-\vect{b} +  \vect{a} - \vect{c}  +\vect{d})}}$};

   \path[->]
        (1') edge
            node[below] {$\geni_i$}
            node[pos=0.05,above = 0.02] {$\vect{a}$}
            node[pos=0.85,above = -0.02] {$\vect{b}$}
            (2');


   \path[->]
        (2') edge[snake it,bend left=41]
        (bp');

\end{tikzpicture}
 \caption{Closed enriched folding}
 \label{fig: closed enriched folding}
 \end{figure}
It is straightforward to see that these two types of enriched foldings do not change the recognized subgroup. Hence, interspersing the appropriate abelian transformations, we can mimic the (any) folding procedure for the skeleton to obtain a reduced enriched automaton recognizing $\HH$ which, after normalizing \wrt a chosen spanning tree \!$\Treei$, will become a Stallings automaton
$\Eti$
for $\HH$.

Note that then
the basepoint subgroup of $\Eti$ is the original basepoint subgroup for $\flower{\HH}$ \emph{possibly enlarged} by the contributions of the eventual closed foldings in the reduction process, whereas
$\sk (\Eti) = \sk(\body{\Eti}) = \stallings{\HH \prF,X}$.
Therefore,
calling $\Basis_{\Treei}
= \set{\,
\llab(\Tpetal{\edgi}) \st \edgi\in \Edgs^{+} \Ati \setmin \Edgs \Treei \,
}
$
(the set of enriched labels of the positive $\Treei$-petals in $\Eti$),
we have $(\Basis_{\Treei}) \prF = \basis_{\Treei}$ (the positive $\Treei$-basis of~$\HH \prF$).
Indeed, besides providing  the desired bijection \eqref{eq: enriched Stallings bijection},
enriched Stallings automata encode the internal structure (and, in particular, a basis) of the subgroups of $\FTA$ in a very transparent way.

\begin{prop}\label{prop: enriched Stallings parts}
Let $\Eti 
$
be a $\Treei$-normalized automaton recognizing $\HH \leqslant \Fta$.
Then, $\HH =\gen{\body{\Eti}} \times L_{\Eti}$,
where $\gen{\body{\Eti}}$ is the image of a splitting of $\prF_{\mid \HH}$, and $L_{\Eti} = \HH \cap \AA$.
Moreover,
$\Basis_{\Treei}$
is a free basis
for~$\gen{\body{\Eti}}$
(called the \defin{(positive) $\Treei$-basis} of $\gen{\body{\Eti}}$)
which, joined to an abelian basis for
$L_{\Eti}$, constitutes a basis for~$\HH$.
\end{prop}

\begin{proof}
The inclusion $L_{\Eti}\leqslant \gen{\Eti}\cap \AA$ is obvious by construction.
For the opposite inclusion,
let $\Basis_{\Treei} = \set{\fta{u_i}{a_i}}_i$, and
suppose that $\ta{a} \in \gen{\Eti}\cap \AA$.
That is, 
$\ta{a} = w(\fta{u_i}{a_i}) \, \ta{l}$,
where $\vect{l} \in L_{\Eti}$,
and $w(\fta{u_i}{a_i})$ denotes a reduced word on the $\fta{u_i}{a_i}$'s.
Since the free part of this element is trivial, and $\set{u_i}_i$ is freely independent,
then $w$ must be the trivial word and thus $\ta{a} = \ta{l} \in L_{\Eti}$, as we wanted to see.
For the second claim, it is enough to consider the homomorphism $\HH \prF \to \HH$ given by $\lab_X(\Tpetal{\edgi}) \mapsto \llab(\Tpetal{\edgi})$, for each arc $\edgi \in \Edgs\Ati \setmin \Treei$, and recall the decomposition~\eqref{eq: subgroup decomposition}.
\end{proof}


\begin{defn}
If $\Eti$ is a $\Treei$-normalized automaton recognizing $\HH \leqslant \FTA$, then
\emph{any} $\Basis_{\Treei}$ defined as above is called an \defin{enriched $\Treei$-basis of $\gen{\body{\Eti}}$}. So, the union of an abelian basis of $L_{\Eti}$, and an enriched $\Treei$-basis of~$\gen{\body{\Eti}}$ is a basis for $\gen{\Eti}$.
\end{defn}

The above considerations, together with the algorithmic nature of bijection \eqref{eq: enriched Stallings bijection}, allow us to easily compute bases of finitely generated subgroups, and  solve the subgroup membership problem within free-times-abelian groups.

\begin{cor} \label{cor: bases are computable}
There exists an algorithm which given a finite family of elements $\mathcal{S} \subseteq \FTA$ outputs a basis for the subgroup $\gen{\mathcal{S}} \leqslant \FTA$.
\end{cor}

\begin{proof}
It is enough to construct a Stallings automaton $\Eti$ for $\gen{\mathcal{S}}$ (normalized \wrt some spanning tree~$\Treei$). Then, an abelian basis for $\gen{\mathcal{S}}\cap \AA = \gen{\Eti} \cap \AA = L_{\Eti}$ can be computed from the generating set at hand, using linear algebra,
whereas an enriched $\Treei$-basis of $\gen{\body{\Eti}}$ is obtained after reading the enriched labels of the $\Treei$-petals in $\Eti$.
\end{proof}

\begin{prop}
The subgroup membership problem is solvable for free-times-abelian groups.
\end{prop}

\begin{proof}
Given $\fta{w}{a} \in \FTA$ and a finite subset $\mathcal{S}\subseteq \FTA$, compute a Stallings automaton $\Eti$ for $\HH=\gen{\mathcal{S}}$. Now, try to realize $w$ as the free label of a \bp-walk in $\sk(\Eti)$: if it is not possible then $w\notin \gen{\sk (\Eti)}=\HH\prF$ and return \nop; otherwise, the enriched label of this \bp-walk provides a vector $\vect{b}\in \AA$ such that $\fta{w}{b}\in \HH$. Finally, $\fta{w}{a} \in \HH$ if and only if $t^{\vect{b}-\vect{a}}\in  L_{\Eti} = \HH\cap \AA$, which is again easily decidable using linear algebra.
\end{proof}

\section{Intersection of subgroups} \label{sec: Intersections FTA}

Intersections of subgroups is a research topic with a long and interesting history. For an arbitrary group $G$, we can consider the following concept and problem as natural starting points.

\begin{defn} \label{def: Howson group}
A group $G$ is said to satisfy the \defin{Howson property} (or to be \defin{Howson} for short) if the intersection of any pair of finitely generated subgroups of $G$ is again finitely generated.
\end{defn}

\begin{named}[Subgroup intersection problem, $\SIP(G)$] \label{def: SIP}
Given two finite sets of words $R,S$ in the generators of $G$, decide whether the intersection $\gen{R} \cap \gen{S}$ is finitely generated; and, in the affirmative case, compute a generating set for the intersection.
\end{named}

It is well known that subgroups of (non-cyclic) finitely generated free groups are again free, but can have any (finite or countably infinite) rank. However, in 1954 \citeauthor{HowsonIntersectionFinitelyGenerated1954} proved that the intersection of two finitely generated subgroups of the free group is always finitely generated; see~\cite{HowsonIntersectionFinitelyGenerated1954}. The classical Stallings automata machinery provides a neat and algorithmic-friendly proof for this remarkable fact, and furthermore makes it possible to compute a basis for the intersection.

\begin{thm}[Howson, \cite{HowsonIntersectionFinitelyGenerated1954}]
\label{thm: Howson}
Free groups are Howson and have solvable \SIP. \qed
\end{thm}

The key concept needed for the geometric proof of this fact is that of product of automata.

\begin{defn} \label{def: product of automata}
Let $\Ati_{\!1},\Ati_{\!2}$ be $X$-automata. The \defin{(tensor or categorical) product} of $\Ati_{\!1}$ and $\Ati_{\!2}$, denoted by $\Ati_{\!1} \times \Ati_{\!2}$, is the automaton with vertex set the Cartesian product $\Verts\Ati_{\!1} \times \Verts \Ati_{\!2}$, an arc $(\verti_1,\verti_2) \xarc{x\,} (\vertii_1,\vertii_2)$ for every pair of arcs $\verti_1 \xarc{x\,} \vertii_1$ in~$\Ati_{\!1}$, and $\verti_2 \xarc{x\,} \vertii_2$ in~$\Ati_{\!2}$ with the same label $x\in X$, and  basepoint~$(\bp_1,\bp_2)$.
\end{defn}

The following easily checkable facts complete the link between intersections of subgroups of the free group and products of Stallings automata.

\begin{lem} \label{lem: product intersection}
If $\Ati_{\!1}$ and $\Ati_{\!2}$ are deterministic $X$-automata, then the product $\Ati_{\!1} \times \Ati_{\!2}$ is again deterministic, and recognizes the intersection of the corresponding subgroups; that is, $\gen{\Ati_{\!1} \times \Ati_{\!2}} = \gen{\Ati_{\!1}} \cap \gen{\Ati_{\!2}}$. \qed
\end{lem}

However, in general, the product of two core automata is not necessarily core (not even connected); so we need to take the core to reach the Stallings automaton of the intersection.

\begin{cor}\label{cor: pb}
Let $H_1,H_2\leqslant \Free[X]$, then $\stallings{H_1\cap H_2,X}=\core (\stallings{H_1,X} \times \stallings{H_2,X})$. \qed
\end{cor}

So, if $H_1$ and $H_2$ are finitely generated, then (from \Cref{thm: Stallings bijection}) $\stallings{H_1,X}$ and $\stallings{H_2,X}$ are finite and computable; hence, $\stallings{H_1 \cap H_2,X}$ is finite and computable too. This proves~\Cref{thm: Howson}.

After Howson's result, the quest for bounds for the rank of the intersection in terms of the ranks of the intersecting subgroups became a popular question in geometric group theory. Concretely, in 1956 \citeauthor{NeumannIntersectionFinitelyGenerated1956} proved that
$\rk(H_1 \cap H_2) - 1 \leq 2(\rk (H_1) - 1) (\rk (H_2) -1)$ for any pair of finitely generated subgroups $1\neq H_1,H_2 \leqslant \Free[X]$, and conjectured that the factor `2' can be removed; see~\cite{NeumannIntersectionFinitelyGenerated1956}.
After many unsuccessful attempts and partial results, two correct (and unrelated) proofs appeared almost simultaneously more than fifty years later (see~\cite{FriedmanSheavesGraphsTheir2015,MineyevSubmultiplicativityHannaNeumann2012} and the remarkable unpublished simplification in \cite{DicksSimplifiedMineyevProof2012}),
and a third one shortly after (see~\cite{Jaikin-ZapirainApproximationSubgroupsFinite2017}).

In \cite{BaumslagIntersectionsFinitelyGenerated1966}, \citeauthor{BaumslagIntersectionsFinitelyGenerated1966} extended Howson's result by showing that the free product of Howson groups is again Howson. However, the same is not true for direct products: Moldavanski (see~\cite{BurnsIntersectionDoubleCosets1998}) already showed that, in $\Free[\set{x,y}] \times \ZZ$, the intersection of the easy looking subgroups $\gen{xt,y}$ and $\gen{x,y}$ is the normal closure of $y$ in $\Free[\set{x,y}]$,
which is not finitely generated; see \Cref{ssec: Moldavanski's example} below for our geometric interpretation of this interesting example. Therefore, in this context the Subgroup Intersection Problem $\SIP(\FTA)$ emerges as a natural and interesting question, specially the decision part (which trivializes in the free case).

The purpose of the present section is to solve $\SIP(\FTA)$ using our enriched version of Stallings automata (\Cref{thm: enriched Stallings bijection}). We approach the problem from a similar perspective to that used in the solution to $\SIP(\Fn)$: in particular, we shall adapt the definition of product of two finite automata to the enriched setting, and obtain an enriched version for \Cref{lem: product intersection}. However, crucial differences must appear with respect to the free case because the situation is intrinsically different, now with $\FTA$ not being Howson.

\begin{defn}\label{def: product of automata2}
Let $\Eti_{\!1} =(\dGri_{\!1}, \vec{\llab}^{\scriptscriptstyle{1}},\bp_1, L_1)$ and $\Eti_{\!2} = (\dGri_{\!2}, \vec{\llab}^{\scriptscriptstyle{2}}, \bp_2, L_2)$ be two enriched automata. Their \emph{product}, denoted by $\Eti_{\!1} \times \Eti_{\!2}$, consists of the product of their respective skeletons $\sk(\Eti_{\!1}) \times \sk (\Eti_{\!2})$ \emph{doubly enriched} with the abelian labeling coming from each factor. That is, for every arc $(\edgi_1,\edgi_2) $ in $\sk(\Eti_{\!1}) \times \sk (\Eti_{\!2})$, and $i =1,2$, we define $\lab_i(\edgi_1,\edgi_2) = (\lab^{\scriptscriptstyle{1}}_i(\edgi_1), \lab^{\scriptscriptstyle{2}}_i(\edgi_2))$; and we attach the pair of subgroups $(L_1,L_2)$ to the basepoint $(\bp_1, \bp_2)$; see~\Cref{fig doubly enriched arc}.
\end{defn}

\begin{figure}[H]
    \centering
\begin{tikzpicture}[shorten >=1pt, node distance=1.75cm and 1.75cm, on grid,auto,>=stealth']
   \node (00) {};
   \node[state,accepting] (0) [below = 0.5 of 00]{};
   \node[] (bp1) [below right =0.1 and 0.15 of 0] {$\scriptscriptstyle{1}$};
   \node[state,accepting] (0') [right = 0.5 of 00]{};
   \node[] (bp1) [below right =0.1 and 0.15 of 0'] {$\scriptscriptstyle{2}$};
   \node[] (L') [left = 0.3 of 0]{$\scriptstyle{L_1}$};
   \node[] (L) [above = 0.3 of 0']{$\scriptstyle{L_2}$};
   \node[state,accepting,blue] (00') [right = 0.5 of 0]{};
   \node[blue] (LL') [right = 0.6 of 00']{$\B{\scriptstyle{(L_1,L_2)}}$};
   \node[state] (1) [below = 1 of 0] {};
   \node[state] (2) [below = of 1] {};
   \node[state] (1') [right = 1 of 0'] {};
   \node[state] (2') [right = 1.75 of 1'] {};
   \node[state,blue] (11') [below right = 1 and 1 of 00'] {};
   \node[state,blue] (22') [below right = of 11'] {};

   \path[dashed]
       (0') edge[]
            (1');
   \path[dashed]
       (0) edge[]
            (1);
   \path[dashed,blue]
       (00') edge[]
            (11');
   \path[->]
       (1') edge[]
            node[pos=0.5,below=-.1mm] {$x$}
            node[pos=0.10,above=-.1mm] {$\vect{a_2}$}
            node[pos=0.83,above=-.1mm] {$\vect{b_2}$}
            (2');
   \path[->]
        (1) edge[]
            node[pos=0.5,right=-.1mm] {$x$}
            node[pos=0.10,left=-.1mm] {$\vect{a_1}$}
            node[pos=0.8,left=-.1mm] {$\vect{b_1}$}
            (2);
   \path[->,blue]
        (11') edge[]
            node[pos=0.5,left] {$x$}
            node[pos=0.05,right=0.3mm] {$(\vect{a_1},\vect{a_2})$}
            node[pos=0.75,right=0.3mm] {$(\vect{b_1},\vect{b_2})$}
            (22');
\end{tikzpicture}
    \caption{Scheme of the product (in blue) of two enriched automata (in black)}
    \label{fig doubly enriched arc}
\end{figure}

So, technically, this product is a $(\AA \!\times \AA)$\!-enriched $X$-automaton with a pair of subgroups of $\AA$ (instead of a subgroup of $\AA \!\times \AA$) attached to the basepoint, a \defin{doubly-enriched automaton}, for short. Walks, and labels of walks in doubly-enriched automata are defined in the natural way. As in the enriched case, the skeleton of a doubly-enriched $X$-automaton is the $X$-automaton obtained after removing from it all the (now double) abelian mass. The notions of connectedness, core, and normalization are extended accordingly.

\begin{rem} \label{rem: strict inclusion}
If $\Eti_1,\Eti_{\!2}$ are Stallings automata recognizing respectively $\HH_1,\HH_2\leqslant \FTA$, then it is clear that $\core (\sk(\Eti_{\!1} \times \Eti_{\!2}))=\core (\sk(\Eti_{\!1})\times \sk(\Eti_{\!2})) = \Stallings {\HH_1\prF \cap \HH_2\prF, X}$. A crucial detail here is that the inclusion $(\HH_1 \cap \HH_2)\prF\leqslant \HH_1\prF \cap \HH_2\prF$ (of subgroups of~$\Free[X]$) is \emph{not} necessarily an equality. Hence, $\core (\sk(\Eti_{\!1} \times \Eti_{\!2}))$ is \emph{not}, in general, equal to $\Stallings {(\HH_1 \cap \HH_2)\prF, X}$. So, further analysis is needed to construct this last automaton, and subsequently $\Stallings {\HH_1 \cap \HH_2, X}$. Observe also that, if $\HH_1, \HH_2$ are finitely generated, then $\HH_1\prF$ and $\HH_2\prF$ (and hence $\HH_1\prF \cap \HH_2\prF$) are so; but $(\HH_1 \cap \HH_2)\prF$ is a (possibly strict) subgroup of the latter, \emph{and may very well not be finitely generated}.
See the characterization in \Cref{prop: fg int characterization FATF}, and Examples \ref{ssec: Moldavanski's example} and \ref{ssec: parameterized example} (Case 2).
\end{rem}

As in the free case, the (core of the) product of enriched automata encodes all the information about the intersection. However, in this case, the resulting doubly-enriched automaton is not a genuine Stallings automaton. Below, we state the enriched version of \Cref{lem: product intersection}, which is clear again by inspection.

\begin{lem} \label{lem: intersection of enriched subgroups}
Let $\Eti_{\!1}=(\dGri_{\!1},\vec{\llab}^{\scriptscriptstyle{1}},L_1)$ and $\Eti_{\!2}=(\dGri_{\!2},\vec{\llab}^{\scriptscriptstyle{2}},L_2)$ be two enriched Stallings automata recognizing the subgroups $\HH_1,\HH_2 \leqslant \Fta = \FTA$, respectively. Then, the intersection $\HH_1 \cap \HH_2$ is precisely the set of elements in $\FTA$ (with free part in $\HH_1 \prF \cap \HH_2 \prF$) that are component-wise readable in the product $\Eti_{\!1} \times \Eti_{\!2}$ modulo the corresponding base subgroups $L_1,L_2$, respectively. More precisely, $\fta{u}{a}$ belongs to $\HH_1\cap \HH_2$ if and only if there is a $(\bp_1, \bp_2)$-walk in $\Eti_{\!1} \times \Eti_{\!2}$ whose label $\fta{u}{(b_1,b_2)}$ satisfies simultaneously $\vect{b_1}-\vect{a}\in L_1$ and $\vect{b_2}-\vect{a}\in L_2$.  \qed
\end{lem}

\begin{defn} \label{def: equalization}
Let $\Eti$ be a doubly enriched automaton with basepoint subgroups $(L_1,L_2)$. We say that $\Eti$ is \defin{equalizable} if the label $\fta{w}{(a,b)}$ of any $\bp$-walk in $\Eti$ satisfies $(\vect{a} + L_1) \cap (\vect{b} + L_2) \neq \varnothing$. Note that, when $\Eti$ is finite, this can be algorithmically tested by normalizing \wrt some previously chosen spanning tree $\Treei$ and, for every arc $\edgi$ with terminal abelian label $(\vect{a},\vect{b}) \neq (\vect{0},\vect{0})$, checking whether $(\vect{a} + L_1) \cap (\vect{b} + L_2) \neq \varnothing$
(this is enough since, after normalization, $(\vect{a},\vect{b})$ is also the abelian label of the petal $\Tpetal{\edgi}$). If $\Eti$ is equalizable, after normalizing \wrt some spanning tree $\Treei$, we can compute a witness $\vect{c} \in (\vect{a} + L_1) \cap (\vect{b} + L_2) $ for each arc outside $\Treei$, and replace the double labeling $(\vect{a},\vect{b})$ with $\vect{c}$;
finally replace the pair of basepoint subgroups $(L_1,L_2)$ with its intersection $L_1 \cap L_2$. The resulting enriched automaton is called the \defin{equalization} of $\Eti$ \wrt \!$\Treei$ (or the \defin{$\Treei$-equalization} of $\Eti$).
\end{defn}

\begin{rem}
The product
$\stallings{\HH_1} \times \Stallings{\HH_2}$
of the Stallings automata of two subgroups
$\HH_1,\HH_2 \leqslant \FTA$ is equalizable if and only if every word in $\HH_1 \prF \cap \HH_2 \prF$ admits compatible completions in $\HH_1$ and $\HH_2$; \ie if for every $w \in \HH_1 \prF \cap \HH_2 \prF$, $ \cab{w}{\HH_1} \cap \cab{w}{\HH_2}  \neq \varnothing$. That is, if and only if
$(\HH_1 \cap \HH_2) \prF = \HH_1 \prF \cap \HH_2 \prF$, which, as explained in \Cref{rem: strict inclusion}, is not always the case.
\end{rem}

Let $\set{\fta{u_{1}}{a_{1,1}},\ldots,\fta{u_{p_1}}{a_{1,p_1}}; \T^{\vect{b_{1,1}}},\ldots,\T^{\vect{b_{1,q_1}}} }$ and $\set{\fta{v_{1}}{a_{2,1}},\ldots,\fta{v_{p_2}}{a_{2,p_2}}; \T^{\vect{b_{2,1}}},\ldots,\T^{\vect{b_{2,q_2}}} }$ be \emph{finite} bases for $\HH_1$ and $\HH_2$, respectively, and let 
$
\basis =\set{w_1, \ldots,w_r}$ be 
a free basis for $\HH_1 \prF \cap \HH_2 \prF$ (all written in terms of the original generators $X,T$ for $\FTA$). This means that, for $i=1,2$, $L_i =\smash{\HH_i \cap \AA = \allowbreak \gen{\T^{\vect{b_{i,1}}},\ldots,\T^{\vect{b_{i,q_i}}}}}$, $\smash{\HH_1 \prF \isom \Free[p_1] = \Free[\set{u_1,\ldots,u_{p_1}}]}$, $\smash{\HH_2 \prF \isom \Free[p_2] =\Free[\set{v_1,\ldots,v_{p_2}}]}$, and $\smash{\HH_1 \prF \cap \HH_2 \prF \isom \Free[r] =\Free[\set{w_1,\ldots,w_{r}}]}$ (note that since both $p_1$ and $p_2$ are finite, $r$ is also finite). Now, consider the following homomorphisms and matrices which compose the diagram in \Cref{fig: intersection diagram big}:
\begin{itemize}
\item $\phi$ (\resp $\phi_1,\phi_2$) is the isomorphism sending each word in $\HH_1 \prF \cap \HH_2 \prF$ (\resp $\HH_1 \prF$, $\HH_2 \prF$)
in the original basis $\set{x_1,\ldots,x_n}$
to its expression in the basis $\set{w_1,\ldots,w_r}$ (\resp $\set{u_1,\ldots,u_{p_1}}$, $\set{v_1,\ldots,v_{p_2}}$);
\item $\rho$ (\resp $\rho_1,\rho_2$) is the abelianization map of $\Free[r]$ (\resp $\Free[p_1],\Free[p_2]$), \emph{not to be confused} with the corresponding restrictions of the global abelianization map $\Free[n]\onto \ZZ^n$;
\item $\matr{B_i}$ is the abelianization of the inclusion map $\HH_1\prF \cap \HH_2\prF \hookrightarrow \HH_i\prF$ (after the change of bases $\phi$ and $\phi_i$), $i=1,2$; note that, although these inclusions are injective maps, the $\matr{B_i}$'s need not be so;
\item $\matr{A_i}$ is the $p_i \times m$ integer matrix having as $j$-th row the vector $\vect{a_{i,j}} \in \AA$, $i=1,2$;
\item $\matr{C_i} \coloneqq \matr{B_i} \matr{A_i}$, $i=1,2$
(where every column of the result must be interpreted modulo the corresponding torsion), and $\matr{D} \coloneqq \matr{C_1}-\matr{C_2}$ is the so-called \defin{difference} matrix.
\end{itemize}

\begin{figure}[H]
\centering
\begin{tikzcd}[row sep=25pt, column sep=0pt,ampersand replacement=\&]
\&\phantom{\isom \Free[\set{x_1,\ldots,x_n}]} \Fn = \Free[\set{x_1,\ldots,x_n}] \& \\[-30pt]
\& \rotatebox[origin=c]{90}{$\leqslant$} \& \\[-28pt]
\HH_1\prF 
\& \HH_1\prF \cap  \HH_2\prF 
\arrow[l,hook'] \arrow[r,hook]
\arrow[rddd,phantom,"\scriptstyle{///}"] \arrow[lddd,phantom,"\scriptstyle{///}"]
    \& \HH_2\prF 
    \\[-28pt]
    \scriptstyle{\phi_1}\ \rotatebox[origin=c]{90}{$\isom$} \ \phantom{\scriptstyle{\phi_1}}
    \&
    \scriptstyle{\phi}\ \rotatebox[origin=c]{90}{$\isom$} \ \phantom{\scriptstyle{\phi}}
    \&
    \scriptstyle{\phi_2}\ \rotatebox[origin=c]{90}{$\isom$} \ \phantom{\scriptstyle{\phi_2}}\\[-28pt]
    \Free[p_1] \arrow[d,->>,"\rho_1"']
    \&
    \Free[r]
    \arrow[d,->>,"\rho"']
    \& \Free[p_2] \arrow[d,->>,"\rho_2"']
    \\
    \ZZ^{p_1} \arrow[rd,->,"\matr{A_1}"']
    \& \ZZ^{r}
    \arrow[d,->,"\matr{D}"] \arrow[l,->,"\matr{B_1}"'] \arrow[r,->,"\matr{B_2}"]\&
    \ZZ^{p_2} \arrow[ld,->,"\matr{A_2}"]\\
    \& \AA\\[-28pt]
    \& \rotatebox[origin=c]{90}{$\leqslant$} \& \\[-28pt]
    \& L_1 \,\leqslant\, L_1 + L_2 \, \geqslant\, L_2 \&
   \end{tikzcd}
   \caption{Intersection diagram I}
   \label{fig: intersection diagram big}
\end{figure}

\begin{rem} \label{rem: case r=0}
The above discussion includes the possibility $r=0$ (corresponding to $\HH_1 \prF \cap \HH_2 \prF = \Trivial$). In this case, $\basis = \varnothing$ and the maps $\rho, \matr{B_1}, \matr{B_2}$ and $\matr{D}$ in \Cref{fig: intersection diagram big} are all trivial.
\end{rem}

\begin{prop} \label{prop: int prF = M preab}
 Let $\HH_1, \HH_2 $ be finitely generated subgroups of $\Fta$. With the above notation,
\begin{equation}\label{eq: int prF = M preab}
(\HH_1 \cap \HH_2) \prF
\,\isom\,
(L_1 + L_2) \matr{D}^{-1} \rho^{-1} \,\normaleq\, \Free[r]\,,
\end{equation}
where $\matr{D} = \matr{B_1} \matr{A_1} - \matr{B_2} \matr{A_2}$, and $\rho\colon \Free[r] \onto \ZZ^{r}$ is the abelianization map; see~\Cref{fig: intersection diagram big,fig: intersection diagram small}.
\end{prop}

\begin{proof}
By definition, $(\HH_1 \cap \HH_2)\prF$ consists exactly of the elements $w\in \HH_1 \prF \cap \HH_2 \prF$ admitting compatible abelian completions in $\HH_1$ and $\HH_2$, \ie such that $\cab{w}{\HH_1}\cap \cab{w}{\HH_2}\neq \varnothing$. On the other side, from \Cref{lem: completion is a coset} and the commutativities in \Cref{fig: intersection diagram big} it is clear that the abelian completion of an element $w \in \HH_1 \prF \cap \HH_2 \prF$ in $\HH_i$ ($i = 1,2$) is $\cab{w}{\HH_i} =w \phi_i \rho_i \matr{A_i} + L_i = w \phi \rho \matr{B_i} \matr{A_i} + L_i$. Hence,
\begin{align*}
    (\HH_1 \cap \HH_2)\prF
    &\,=\,
    \set{w \in \HH_1 \prF \cap \HH_2\prF \st \cab{w}{\HH_1} \cap \cab{w}{\HH_2}\neq \varnothing
    }\\
    &\,=\,
    \set{w \in \HH_1 \prF \cap \HH_2\prF \st
    (w \phi \rho \matr{B_1} \matr{A_1} + L_1) \cap (w \phi \rho \matr{B_2} \matr{A_2} + L_2)
    \neq \varnothing
    }\\
    &\,=\,
    \set{w \in \HH_1 \prF \cap \HH_2\prF \st
    w \phi \rho (\matr{B_1} \matr{A_1} - \matr{B_2} \matr{A_2}) \in L_1 + L_2
    }\\
    &\,=\,
    (L_1 + L_2) (\matr{B_1} \matr{A_1} - \matr{B_2} \matr{A_2})^{-1} \rho^{-1} \phi^{-1} \\
    &\,\isom\,
    (L_1 + L_2) \matr{D}^{-1} \rho^{-1} \, .
\end{align*}
Finally, the normality of $(L_1 + L_2) \matr{D}^{-1} \rho^{-1}$ in $\Free[r]$ follows immediately from the abelianity of $(L_1 + L_2)\matr{D}^{-1}$ and the surjectivity of the abelianization $\rho$.
\end{proof}

The key point in~\Cref{eq: int prF = M preab} is that it allows to express $(\HH_1 \cap \HH_2) \prF$ (and so, its finitely generated character
) in abelian terms. Now, we are ready to establish the claimed link between Stallings automata and Cayley digraphs of abelian groups. Recall that the vertical inclusions between the two rows in \Cref{fig: intersection diagram small} are all normal (since $\AA$ and $\ZZ^{r}$ are abelian, and $\rho$ is onto).
\begin{figure}[H]
\centering
\begin{tikzcd}[row sep=-3pt, column sep=17pt]
   \Free[n] \, \geqslant \
   \HH_1 \prF \cap \HH_2 \prF 
   \ \overset{\scriptstyle{\phi}}{\isom}\hspace{-50pt}
   & \Free[r]  \arrow[r,->>,"\rho"]
   & \ZZ^{r}  \arrow[r,"\matr{D}"] & \AA\\
   & \rotatebox[origin=c]{90}{$\normaleq$}&\rotatebox[origin=c]{90}{$\normaleq$}
   &\rotatebox[origin=c]{90}{$\normaleq$}\\[-5pt]
   (\HH_1 \cap \HH_2) \prF \ \overset{\scriptstyle{\phi}}{\isom} \hspace{-22pt}
   & \underbrace{(L_1 + L_2) \matr{D}^{-1} \rho^{-1}}_{M \rho^{-1}} 
   & \underbrace{(L_1 + L_2) \matr{D}^{-1}}_{M} \arrow[l,mapsto] & L_1 + L_2 \arrow[l,mapsto]
   \end{tikzcd}
    \caption{Intersection diagram II}
    \label{fig: intersection diagram small}
\end{figure}
Defining $M \coloneqq (L_1 + L_2) \matr{D}^{-1}\leqslant \ZZ^r$,
$s \coloneqq \rk(M)\leq r$,
and taking the respective quotient groups, we have
\begin{equation} \label{eq: chain of isomorphisms of quotient groups}
\HH_1\prF \cap \HH_2 \prF / (\HH_1 \cap \HH_2) \prF
\, \overset{\overline{\phi}}{\isom}\,
\Free[r] / M \rho\preim
\, \overset{\overline{\rho}}{\isom}\,
\ZZ^{r} / M \,.
\end{equation}
We call
$\matr{M}$ the $s\times r$ integer matrix having as rows the elements of some abelian basis for $M$
(note that $s\leq r$).
Then, the \defin{Smith normal form} of $\matr{M}$ is an integral $s\times r$ matrix $\matr{S}=\diag(\diagi_1,\ldots,\diagi_s)$, where $\diagi_1,\ldots,\diagi_s\in \ZZ\setminus \{0\}$, $\diagi_1 |\cdots |\diagi_s$, and $\matr{P}$ and $\matr{Q}$ are invertible matrices ($\matr{P}\in GL_{s}(\ZZ)$, $\matr{Q}\in GL_{r}(\ZZ)$) such that $\matr{P}\matr{M}\matr{Q}=\matr{S}$.
If we finally define
$\diagi_i \coloneqq 0$ for each $i=s+1,\ldots ,r$ (in case they exist), then
\begin{align} \label{eq: chain of isomorphisms of fg quotient groups}
\HH_1\prF \cap \HH_2 \prF / (\HH_1 \cap \HH_2) \prF
\,\isom\,
\ZZ^r /\gen{\matr{S}}
\,=\,
\textstyle{\bigoplus_{i=1}^{r} \, \ZZ /\diagi_i \ZZ}
\,=\,
(\bigoplus_{i=1}^{s} \, \ZZ /\diagi_i \ZZ)\,\oplus\, \ZZ^{r-s} \,.
\end{align}
Furthermore, the index of $(\HH_1 \cap \HH_2) \prF$ in $\HH_1 \prF \cap \HH_2 \prF$ is
\begin{equation} \label{eq: index = product deltas}
\Ind{\HH_1 \prF \cap \HH_2 \prF }{(\HH_1 \cap \HH_2) \prF}
  \,=\,
  \textstyle{\prod_{i=1}^{r} \,\Ind{\ZZ}{\diagi_i \ZZ} }
  \,=\,
  \left\{\!
  \begin{array}{ll}
  \diagi_1  \cdots  \diagi_s < \infty &\text{if } s=r\,, \\[3pt]
  \infty &\text{if } s<r \,.
  \end{array}
  \right.
\end{equation}
This allows us to interpret the Stallings automaton of $(\HH_1 \cap \HH_2)\prF$ as the Cayley multidigraph (a generalization of the classical Cayley digraph allowing repeated generators; see the precise definition below) of the finitely generated abelian group in \cref{eq: chain of isomorphisms of quotient groups}; and ultimately, relate the rank of the intersection $\HH_1 \cap \HH_2$ to the index of $M$ in~$\ZZ^r$.

\begin{defn}
Let $G$ be a group and let $\mset{h_i}_{i\in I}$ be a multiset of generators for $G$ (\ie a set of generators with possible repetitions). Then, the \defin{Cayley multidigraph} of $G$ \wrt $\mset{h_i}_{i\in I}$, denoted by $\cayley{G,\mset{h_i}_{i\in I}}$, is the multidigraph with vertex set $G$, and an $h_i$-arc $g \xarc{h_i\,} g h_i$ for every $g \in G$, and every $i\in I$. It is allowed that, for some $i\in I$, $h_i$ is the trivial element, hence producing loops labeled by $h_i$ in every vertex. Of course, if $\mset{h_i}_{i\in I}$ is a set, then $\cayley{G,\mset{h_i}_{i\in I}}$ is the standard Cayley digraph of $G$.
\end{defn}

\begin{thm}\label{thm: Stallings = Cayley}
Let $\HH_1,\HH_2$ be two finitely generated subgroups of $\Fta$.
Then, either $(\HH_1 \cap \HH_2)\prF$ is trivial, or
(with the above notation)
\begin{equation} \label{eq: Stallings = Cayley}
 \Stallings{(\HH_1 \cap \HH_2) \prF \,, \basis}
    \,\isom\,
    \Cayley{
    \textstyle{\bigoplus_{i=1}^{r} \ZZ/\diagi_i \ZZ} \,, \mset{\vect{e}_i \matr{Q}}_i }
    \, ,
\end{equation}
where $\basis =\set{w_1, \ldots, w_r}$
is a (finite) free basis for $\HH_1\prF \cap \HH_2\prF$, $\set{\vect{e}_1
,\ldots ,\vect{e}_r}$ is the canonical basis of $\ZZ^{r}$,
and $\mset{\vect{e}_i \matr{Q}}_{i=1,\ldots ,r}$ is the multiset consisting of the rows of $\matr{Q}$ (recall that $\matr{S} = \matr{P} \matr{M} \matr{Q}$)
interpreted as elements of $\bigoplus_{i=1}^{r} \ZZ / \diagi_i \ZZ$.
\end{thm}

\begin{rem}
Note that the generators $\vect{e_i} \matr{Q}$ in~\eqref{eq: Stallings = Cayley} must be interpreted as elements in an
ordered multiset (in order to keep track of the link between generators in the corresponding automata).
\end{rem}

\begin{rem}
Most of (the non-algorithmic part of) the analysis started in \Cref{fig: intersection diagram big} is still valid for arbitrary (maybe non finitely generated) subgroups $\HH_1,\HH_2 \leqslant \FTA$. Then $p_1,p_2$ and $r$ may be infinite, but \Cref{eq: int prF = M preab,eq: chain of isomorphisms of quotient groups} are still valid (with the natural definition of $\matr{D}$ as an $\infty \times m$ integer matrix), and we can rephrase \Cref{thm: Stallings = Cayley} saying that $\stallings{(\HH_1 \cap \HH_2)\prF,\set{w_1,\dots,w_r}}$ is isomorphic to the corresponding Cayley multidigraph of a countably generated abelian group.
\end{rem}

\begin{proof}[Proof of \Cref{thm: Stallings = Cayley}]
Assume $(\HH_1 \cap \HH_2) \prF \neq \Trivial$; in particular, $M\rho\preim\neq \Trivial$, $\HH_1\prF \cap \HH_2\prF \neq \Trivial$,  and $r\neq 0$; put $I=\set{1,\ldots ,r}$. The claimed result follows from the following chain of equalities and automata isomorphisms:
 \begin{align} 
 \Stallings{(\HH_1 \cap \HH_2) \prF \,,\, \set{w_i }_{i\in I}}
    & \label{eq: St(cap) = St(Mpreab)}
    \,\isom \,\stallings{M\rho\preim, \set{w_i \phi}_{i\in I}} \\[3pt]
    & \label{eq: Sch}
    \,=\,  \schreier{{M\rho\preim, \set{w_i \phi}_{i\in I}}} \\[3pt]
    & \label{eq: Cay Fr}
    \,=\,  \cayley{\Free[ \set{w_i}_{i\in I}] / M\rho\preim, \mset{w_i \phi \cdot (M\rho \preim)}_{i\in I}} \\[3pt]
    & \label{eq: Cay Z/MM}
    \,\isom\,  \cayley{\ZZ^r / \gen{\matr{M}}, \mset{\vect{e}_i + \gen{\matr{M}}}_{i\in I}} \\[3pt]
    & \label{eq: Cay Z/DDQQ}
    \,=\,  \cayley{\ZZ^r / \gen{\matr{S} \, \matr{Q}^{-1}}, \mset{\vect{e}_i + \gen{\matr{M}}}_{i\in I}} \\[3pt]
    & \label{eq: Cay Z/DD}
    \,\isom\,  \cayley{\ZZ^r / \gen{\matr{S}} \,,\, \mset{\vect{e}_i \matr{Q} + \gen{\matr{S}}}_{i\in I}} \\[3pt]
    & \label{eq: Cay Z/delta}
    \,=\,  \cayley{\textstyle{\bigoplus_{i=1}^{r} \ZZ/\diagi_i \ZZ} \,, \mset{\vect{e}_i \matr{Q}}_{i\in I}} \,.
  \end{align}
The isomorphism \eqref{eq: St(cap) = St(Mpreab)} follows immediately from \eqref{eq: int prF = M preab}. The equalities \eqref{eq: Sch} and \eqref{eq: Cay Fr} are consequences of the normality of $M\rho \preim$ in $\Free[r]$ (note that \eqref{eq: Sch} also needs the assumed condition $M\rho^{-1} \neq \Trivial$). Observe that different $w_i$'s may result in the same coset modulo $M\rho\preim$; this is why $\mset{w_i \phi \cdot (M\rho\preim)}_{i\in I}$, and the subsequent ones in \Cref{eq: Cay Z/MM,eq: Cay Z/MM,eq: Cay Z/DDQQ,eq: Cay Z/DD,eq: Cay Z/delta} must be understood as multisets. The isomorphism \eqref{eq: Cay Z/MM}
(where $\vect{e_i} = w_i \phi \rho $)
is clear from the (group) isomorphism $\overline{\rho}$ in \eqref{eq: chain of isomorphisms of quotient groups}.

Now compute a basis for $M=(L_1 +L_2)\matr{D}\preim$ from the starting data, and write it in the rows of an $s\times r$ integral matrix $\matr{M}$, where $0\leq s=\rk(M)\leqslant r$. Then, compute its Smith normal form $\matr{S}=\diag(\diagi_1,\ldots,\diagi_s)$ together with the invertible matrices $\matr{P}$ and $\matr{Q}$ such that $\matr{P}\matr{M}\matr{Q}=\matr{S}$. Since $\matr{P}$ is invertible, it is clear that $M=\gen{\matr{M}}=\gen{\matr{P}^{-1} \matr{S} \matr{Q}^{-1}}=\gen{\matr{S} \matr{Q}^{-1}}$ and 
\eqref{eq: Cay Z/DDQQ} follows. Finally, applying the automorphism $\matr{Q}\colon \ZZ^r \to \ZZ^r$ to both the group elements and the arc labels, we obtain the isomorphism \eqref{eq: Cay Z/DD} which with the convention that $\diagi_i = 0$ for $i=s+1,\ldots ,r$ takes the form~\eqref{eq: Cay Z/delta}.
\end{proof}

Of course, the situation is special in the degenerate case $(\HH_1 \cap \HH_2) \prF =\Trivial$. The following lemma clarifies the distinction between the two cases.

\begin{lem} \label{lem: int pr = 1}
Let $\HH_1,\HH_2$ be finitely generated subgroups of $\Fta$, and let $r$ denote the (finite) rank of $\HH_1 \prF \cap \HH_2 \prF$.
\begin{enumerate}[ind]
\item \label{item: r=0} If $r=0$, then $(\HH_1 \cap \HH_2) \prF =\Trivial$.
\item \label{item: r=1} If $r=1$, then $(\HH_1 \cap \HH_2) \prF = \Trivial$ if $\diagi_1 = 0$, and $\rk ((\HH_1 \cap \HH_2) \prF) =1$ otherwise.
\item \label{item: r>1} If $r\geq 2$, then $[\HH_1 \prF \cap \HH_2 \prF, \allowbreak
\HH_1 \prF \cap \HH_2 \prF] \leqslant (\HH_1 \cap \HH_2)\prF \neq \Trivial$, and
\begin{equation} \label{eq: rk int prF}
\rk \left((\HH_1 \cap \HH_2) \prF \right) -1
\,=\,
\frac{\diagi_1  \cdots \diagi_s}{\diagi_{s+1}  \cdots \diagi_r}  \cdot
\left(r-1\right) .
\end{equation}
\end{enumerate}
In particular,
$(\HH_1 \cap \HH_2)\prF = \Trivial$
if and only if
either $r=0$, or both $r=1$ and $M = \set{\vect{0}}$.
\end{lem}

\begin{proof}
\ref{item: r=0}~The case $r=0$ is trivial.
\ref{item: r=1}~If $r=1$, then the (cyclic) subgroup $(\HH_1 \cap \HH_2) \prF$ is trivial if and only if $M=\set{\vect{0}}$ or, equivalently, $\diagi_1=0$.
\ref{item: r>1}~Firstly note that $(\HH_1 \cap \HH_2)\prF \isom (L_1 + L_2) \matr{D}\preim \rho \preim  = M \rho \preim \geqslant \Der{\Free[r]}$, which is non-trivial when $r\geq 2$. Then, \eqref{eq: rk int prF} follows easily from~\Cref{eq: index = product deltas}: if the index $\ind{\HH_1\prF \cap \HH_2 \prF}{(\HH_1 \cap \HH_2) \prF}$ is finite, then \eqref{eq: rk int prF} corresponds precisely to the well-known Schreier index formula. Otherwise, $(\HH_1 \cap \HH_2) \prF$ is a non-trivial normal subgroup of infinite index in $\HH_1\prF \cap \HH_2 \prF $ and hence has infinite rank; and, on the other hand, $s<r$ and the right hand side of~\eqref{eq: rk int prF} is infinite as well. The last claim is obvious from the above discussion.
\end{proof}

A neat characterization of when the intersection of two finitely generated subgroups of $\FTA$ is again finitely generated follows easily from \Cref{thm: Stallings = Cayley} and the previous considerations. Note that, since the parameters $r$ and $s$ in~\Cref{prop: fg int characterization FATF}\ref{item: r=0 or r=s} are clearly computable, this immediately solves the decision part of $\SIP(\FTA)$.

\begin{prop} \label{prop: fg int characterization FATF}
Let $\HH_1,\HH_2$ be finitely generated subgroups of $\Fta$. Then, the following conditions are equivalent:
\begin{enumerate}[dep]
\item \label{item: int fg} the intersection $\HH_1 \cap \HH_2$ is finitely generated;
\item \label{item: int pi fg} the projection $(\HH_1 \cap \HH_2)\prF$ is finitely generated;
\item \label{item: r=0 or r=s}  either $r=0$, $r=1$, or $2 \leq r = s$; 
\item \label{item: fi or trivial} the (normal) subgroup $(\HH_1 \cap \HH_2)\prF$ is either trivial, or has finite index in $\HH_1\prF \cap \HH_2 \prF$.
\end{enumerate}
\end{prop}

\begin{proof}[Proof]
[\ref{item: int fg}$\Biimp$\ref{item: int pi fg}] This is a particular instance of~\Cref{cor: H fg iff Hpi fg}.

[\ref{item: int pi fg}$\Biimp$\ref{item: r=0 or r=s}]
If $r=0$ or $1$, then every subgroup of $\HH_1 \prF \cap \HH_2 \prF$ is cyclic and hence finitely generated.
Otherwise
(since $(\HH_1  \cap \HH_2) \prF \neq \Trivial$),
\Cref{eq: Stallings = Cayley} holds and thus $(\HH_1 \cap \HH_2) \prF$ is finitely generated if and only if the group $\bigoplus_{i=1}^{r} \ZZ/ \diagi_i \ZZ$ is finite, which happens if and only if $s=\rk(M)=r$.

[\ref{item: r=0 or r=s}$\Biimp$\ref{item: fi or trivial}] From \Cref{thm: Stallings = Cayley} and \Cref{eq: chain of isomorphisms of quotient groups} (see \Cref{eq: Cay Fr} in the proof), if $(\HH_1 \cap \HH_2) \prF \neq \Trivial$ then $(\HH_1 \cap \HH_2) \prF$ is finitely generated if and only if the index $\ind{\HH_1 \prF \cap \HH_2 \prF}{(\HH_1 \cap \HH_2) \prF}$ is finite.
\end{proof}

\begin{rem} \label{rem: L1=L2=0}
Suppose $\HH_1, \HH_2\leq \FTA$ both have trivial abelian part, namely $L_1=\HH_1\cap \AA=\set{\vect{0}}$ and $L_2=\HH_2\cap \AA=\set{\vect{0}}$. In this case, $M=(L_1 +L_2)\matr{D}\preim =\set{\vect{0}}\,\matr{D}\preim =\ker \matr{D}$ is a direct summand of $\ZZ^r$. So, either $M=\ZZ^r$ (and so, $(\HH_1\cap \HH_2)\prF =\HH_1\prF \cap \HH_2\prF$) or $M$ is of infinite index in $\ZZ^r$. Hence, in this case, $(\HH_1\cap \HH_2)\prF$ is finitely generated if and only if it equals $\HH_1\prF \cap \HH_2\prF$.
\end{rem}

Finally, we can combine the developed machinery to compute an enriched Stallings automaton for $\HH_1 \cap \HH_2$.

\begin{thm} \label{thm: enriched Stallings computable}
With the above notation, and after detecting that $\HH_1 \cap \HH_2$ is finitely generated, the following procedure outputs a Stallings automaton for $\HH_1 \cap \HH_2$:
\begin{enumerate}[seq]
\item \label{item: draw Stallings cap pi}
Compute the Stallings automaton $\Atii$ of $(\HH_1 \cap \HH_2)\prF$ \wrt a  free basis $\set{w_1, \ldots, w_r}$ for $\HH_1 \prF \cap \HH_2 \prF$.
\item \label{item: replace}
Replace each $w_i$-arc in $\Atii$ by a directed $X$-path spelling $w_i=w_i(X) \neq \trivial$, doubly-enriched with a pair of vectors $(\vect{a_i},\vect{b_i}) \in \AA \times \AA$
(attached, say, to the end of the last arc)
such that $\fta{w_i}{a_i}\in \HH_1$ and $\fta{w_i}{b_i}\in \HH_2$;  and attach the pair of subgroups $(L_1,L_2)$ to the basepoint.
\item \label{item: foldings} Reduce the resulting automaton until a reduced doubly-enriched automaton is obtained.
\item \label{item: equalize}
Equalize the automaton obtained \wrt a chosen spanning tree.
\end{enumerate}
\end{thm}

\begin{proof}
We start by computing the Stallings automata $\Eti_{\!1} = \stallings{\HH_1,X}$ and $\Eti_{\!2} = \stallings{\HH_2,X}$ (see~\Cref{thm: enriched Stallings bijection}). In particular, we can use linear algebra to obtain abelian-bases for the subgroups $L_1=\HH_1\cap \AA$, $L_2=\HH_2\cap \AA$, and hence an abelian basis 
for the subgroup $L_1 \cap L_2 = \HH_1 \cap \HH_2 \cap \AA$. Also, we choose spanning trees and compute the corresponding bases for $\HH_1$ and $\HH_2$.

Then, compute the doubly-enriched automaton $\Eti_{\!1} \times \Eti_{\!2}$, $T$-normalized \wrt a chosen spanning tree $T$; compute the corresponding free basis $\basis_T=\set{w_1, \ldots ,w_r}$ for $\HH_1\prF\cap \HH_2\prF$, and let $\Atiii = \stallings{\HH_1\prF \cap \HH_2\prF,X}=\core (\sk(\Eti_{\!1}) \times \sk(\Eti_{\!2}))$; see~\Cref{cor: pb}. Finally, we compute the integral matrices $\matr{A_1}$, $\matr{A_2}$, $\matr{B_1}$, $\matr{B_2}$ and $\matr{D}$ (see~\Cref{fig: intersection diagram big}), and an abelian basis for the subgroup $M=(L_1+L_2)\matr{D}\preim\leqslant \ZZ^r$, which we write in the rows of a new integral matrix $\matr{M}$ of size $s\times r$, where $s=\rk(M)\leqslant r=\rk(\HH_1\prF \cap \HH_2\prF)$. Now, let us distinguish two cases:

If the automaton $\stallings{\HH_1\prF \cap \HH_2\prF, X}$ is just a point
(so, $r=\rk(\HH_1\prF \cap \HH_2\prF)=0$), or it has rank $r=1$ but $M=\set{\vect{0}}$,
then, by~\Cref{lem: int pr = 1}, $(\HH_1 \cap \HH_2) \prF $ is trivial and hence finitely generated.
In this case, $\Stallings{\HH_1\cap \HH_2,X}$ is a single point $\bp$ with attached subgroup $L_1\cap L_2$.

Otherwise,
$1\leq s=r$ (since we are assuming $(\HH_1\cap \HH_2) \prF$ is finitely generated)
and we can apply \Cref{thm: Stallings = Cayley}: compute the Smith normal form for $\matr{M}$, say $\matr{S}=\diag(\diagi_1,\ldots,\diagi_r)$, where $\diagi_1,\ldots,\diagi_r\in \ZZ\setminus \{0\}$, $\diagi_1 |\cdots |\diagi_r$, together with invertible matrices $\matr{P},\, \matr{Q}\in GL_{r}(\ZZ)$ such that $\matr{P}\matr{M}\matr{Q}=\matr{S}$, and draw the Cayley multidigraph indicated in~\cref{eq: Stallings = Cayley}, corresponding to the finite abelian group $\bigoplus_{i=1}^{r} \ZZ/\diagi_i \ZZ$. After reinterpreting the labels accordingly, this is nothing else but the Stallings automaton $\Atii$ of $(\HH_1 \cap \HH_2)\prF$ as a subgroup of $\HH_1 \prF \cap \HH_2 \prF$ and \wrt the ambient free basis $\set{w_1, \ldots, w_r}$. This is the content of step \ref{item: draw Stallings cap pi}.

Note that each generator $w_i$ corresponds to an edge $\edgi_i$ in $\Ati_1 \times \Ati_2$ outside $T$ with a double label $\lab_2(\edgi_i)=(\vect{a_i},\vect{b_i})$ and closing a $(\bp_1, \bp_2)$-walk
\smash{$\Tpetal{\edgi_i}$}
with label $\fta{w_i}{(a_i,b_i)}$, such that $\fta{w_i}{a_i}\in \HH_1$ and $\fta{w_i}{b_i}\in \HH_2$. After replacing every $w_i$-arc in $\Atii$ with the doubly enriched $X$-path
\smash{$\Tpetal{\edgi_i}$},
successively folding the resulting automaton, and finally taking the core, we obtain a reduced doubly-enriched $X$-automaton $\Atii'$ such that its free part recognizes $\gen{\sk(\Atii')} = (\HH_1 \cap \HH_2)\prF$, and when read \wrt the first (\resp second) abelian components recognizes a subgroup of $\HH_1$ (\resp $\HH_2$). Note that no closed foldings are involved, since $\rk (\Atii) = \rk (\HH_1 \cap \HH_2)\prF =\rk (\Atii')$, so no vector gets added to the basepoint subgroups, which remain equal to $L_1$ and $L_2$. This is the content of steps \ref{item: replace} and \ref{item: foldings}.

According to \Cref{def: equalization}, step \ref{item: equalize} consists of three parts. Firstly, normalize $\Atii'$ \wrt some chosen spanning tree $\TT$ (that is, use abelian transformations to concentrate the double abelian mass of $\Atii'$ into the heads of the edges outside~$\TT$). Secondly, for every edge outside $\TT$, read the corresponding label $\fta{w}{(a,b)}$. By construction, $\fta{w}{a}\in \HH_1$ and $\fta{w}{b}\in\HH_2$, but also $w\in (\HH_1\cap \HH_2)\prF$ and so, the coset intersection $(\vect{a}+L_1)\cap (\vect{b}+L_2)$ is non-empty; this means that $\Atii$ is equalizable. Compute $\vect{c}\in (\vect{a}+L_1)\cap (\vect{b}+L_2)$ and replace in $\Atii'$ the double labeling $(\vect{a},\vect{b})$ with the genuine one $\vect{c}\in \AA$. Finally, replace $(L_1,L_2)$ by $L_1\cap L_2$ as basepoint subgroup, and call $\Ati$ the final obtained automaton. This is the equalization process mentioned in step \ref{item: equalize}.

By construction, $\Ati$ is an enriched,
reduced, and $\TT$-normalized automaton such that $\gen{\Ati} \leqslant \HH_1 \cap \HH_2$
and $\gen{\sk (\Ati)} = (\HH_1 \cap \HH_2) \prF$.
Moreover, given an element $\fta{u}{d} \in \HH_1 \cap \HH_2$, $u\in (\HH_1 \cap \HH_2)\prF$ and so it is the free label of a $\bp$-walk in $\Ati$. This walk reads an element $\fta{u}{e} \in \gen{\Eti} \leqslant \HH_1 \cap \HH_2$; hence, $\vect{d} -\vect{e} \in \HH_1 \cap \HH_2 \cap \AA = L_1 \cap L_2$, and so $\fta{u}{d} \in \gen{\Ati}$. Therefore,
$\gen{\Ati} = \HH_1 \cap \HH_2$ and $\Ati$ is a Stallings automaton for $\HH_1 \cap \HH_2$.
\end{proof}

Since finite Stallings automata provide computable bases for the subgroups they recognize, the above results immediately solve the $\SIP$ for free-times-abelian groups.

\begin{cor} \label{cor: SIP FTA is solvable}
The subgroup intersection problem $\SIP(\Fta)$ is solvable. \qed
\end{cor}

The computability part of the $\SIP$ problem refers to the case where the intersection $\HH_1\cap\HH_2$ is finitely generated. We claim that, even when it is not, we can also ``compute'' a basis for $\HH_1\cap\HH_2$. It is not clear whether the above proof given for the finitely generated case generalizes to a recursive construction since one would have to do a similar procedure with increasing finite pieces of the (now infinite) Cayley graph from~\Cref{thm: Stallings = Cayley}, and then somehow control or bound the effect of the foldings coming from new additions onto the previously computed part.
Instead, we present an alternative  approach covering both the finite and the infinite cases, and providing the desired result.
The new key concept needed is that of \defin{vertex expansion}, which we present below.

\begin{defn}
Let $\Atiii$ be a reduced doubly-enriched automaton normalized \wrt a spanning tree $\Treei$, and let $\basis$ be the corresponding basis for $\gen{\sk (\Atiii)}$. Then, given a \mbox{$\basis$-automaton~$\Atii$}, we define the \defin{vertex expansion} of $\Atii$ by 
$\Atiii$ \wrt $\Treei$
as the doubly-enriched automaton
$\expand{\Atii}{\Atiii,\!\Treei}$ obtained in the following way:
\begin{enumerate}
\item Replace every vertex $\verti$
in $\Atii$ by a copy $\Treei^{(\verti)}$ of the $X$-labeled tree $\Treei$ (and denote by $\graphstyle{v}^{(\verti)}$ the copy of the vertex $\graphstyle{v} \in \Verts \Treei$ in $\Treei^{(\verti)}$).
\item For $w\in \basis$, replace every $w$-arc $\edgi \equiv \verti \xarc{w\,} \vertii$ in $\Atii$ by an arc $\init^{(\verti)} \xarc{\ \,} \term^{(\vertii)}$ in $\expand{\Atii}{\Atiii,\!\Treei}$, where
$w$ is the label of the $\Treei$-petal $\bp \xleadsto{_{\scriptscriptstyle{\Treei}}} \init \xarc{\ \,} \term \xleadsto{_{\scriptscriptstyle{\Treei}}} \bp$ in $\Atiii$.
\item Label each new edge $\init^{(\verti)} \xarc{\ \,} \term^{(\vertii)}$ by the full labeling $\llab(\init \xarc{\ \,} \term)$ from $\Atiii$.
\item Declare $\bp^{(\bp)}$ as basepoint, with attached basepoint subgroups $(L_1,L_2)$ as in $\Atiii$.
\end{enumerate}
\end{defn}

\begin{rem}
There are natural correspondences between \bp-walks in $\expand{\Atii}{\Atiii,\!\Treei}$ and \bp-walks in $\Atii$ and $\Atiii$, which preserve free labels: first, note that any \bp-walk in $\expand{\Atii}{\Atiii,\!\Treei}$ translates, verbatim, into a \bp-walk in $\Atiii$ with the same enriched label. Second, note that $\Atii$ can be recovered from $\expand{\Atii}{\Atiii,\!\Treei}$ by collapsing back every copy $\Treei^{(\verti)}$ to the vertex $\verti$ in $\Atii$. Hence, every \bp-walk in $\expand{\Atii}{\Atiii,\!\Treei}$ projects to a \bp-walk in $\Atii$, and every \bp-walk in $\Atii$ elevates to a uniquely determined \bp-walk in $\expand{\Atii}{\Atiii,\!\Treei}$. Moreover, it is clear that both transformations preserve labels as elements in $\Fn \times \overline{\ZZ}^{m + m}$.
\end{rem}

\begin{prop} \label{prop: enriched Stallings II}
Let $\HH_1,\HH_2$ be finitely generated subgroups of $\FTA$ with respective Stallings automata $\Ati_1,\Ati_2$. Let $\Atiii$ be the core of $\Ati_1 \times \Ati_2$ normalized \wrt some spanning tree~$\Treei$, and  let $\Atii = \Stallings{(\HH_1 \cap \HH_2) \prF \,,\, \basis_{\Treei}}$ (see \Cref{thm: Stallings = Cayley}). Then, the vertex expansion $\expand{\Atii}{\Atiii,\!\Treei}$ is a doubly enriched, reduced, and equalizable automaton which,
after equalizing,
constitutes a Stallings automaton for $\HH_1 \cap \HH_2$.
\end{prop}

\begin{proof}
It is enough to see that $\expand{\Atii}{\Atiii,\!\Treei}$ is deterministic, core, equalizable, and, furthermore, after equalization, it recognizes the intersection $\HH_1 \cap \HH_2$. The determinism of $\expand{\Atii}{\Atiii,\!\Treei}$ is clear from the determinism of $\Atii$ and $\Atiii$ (and hence of $\Treei$). Second, it is easy to see that, since $\Atiii$ is core and $\Atii$ is core and saturated, $\expand{\Atii}{\Atiii,\!\Treei}$ is also core.

Now, let us see that $\expand{\Atii}{\Atiii,\!\Treei}$ is equalizable. Let $\walki$ be an arbitrary \bp-walk in $\expand{\Atii}{\Atiii,\!\Treei}$ and consider its label $\llab(\walki) = \fta{w}{(a,b)}$. Note that, by construction, the projection of $\walki$ to $\Atii$ is a \bp-walk reading the same $w \in \Fn$; therefore  $w \in \gen{\Atii} = (\HH_1 \cap \HH_2) \prF$, and hence there exist $\vect{c} \in \AA$ such that $\fta{w}{c} \in \HH_1 \cap \HH_2$. On the other hand, $\walki$ is verbatim a $\bp$-walk in $\Atiii$ with exactly the same doubly-enriched label $\fta{w}{(a,b)}$; hence, $\fta{w}{a} \in \HH_1$ and $\fta{w}{b} \in \HH_2$. Therefore, $\vect{c} \in (\vect{a} + L_1) \cap (\vect{b} + L_2) \neq \varnothing$, and $\expand{\Atii}{\Atiii,\!\Treei}$ is equalizable as claimed.

Finally, choose a spanning tree~$\TT$, and ($\TT$-normalize and) equalize $\expand{\Atii}{\Atiii,\!\Treei}$ \wrt it; denote $\Ati$ the resulting enriched Stallings automaton (with basepoint subgroup $L_{\Eti} = L_1 \cap L_2$). If $\fta{w}{c} \in \gen{\Ati}$, then $\fta{w}{(c,c)} \in \gen{\expand{\Atii}{\Atiii,\!\Treei}}$ and, from the paragraph above, $\fta{w}{c} \in \HH_1 \cap \HH_2$; hence, $\gen{\Ati} \leqslant \HH_1 \cap \HH_2$. Conversely, if $\fta{w}{d} \in \HH_1 \cap \HH_2$ then $w\in (\HH_1 \cap \HH_2) \prF$ and so it is the free label of some\! \bp-walk $\walki$ in $\Atiii$; then, the label of $\walki$ viewed as a $\bp$-walk in $\Ati$ is $\fta{w}{c}$, for some $\vect{c} \in \AA$, that is, $\fta{w}{c} \in \gen{\Ati} \leqslant \HH_1 \cap \HH_2$. Therefore $\vect{d}- \vect{c} \in L_1 \cap L_2$ and so $\fta{w}{d} \in \gen{\Ati}$. This shows that $\gen{\Ati} = \HH_1 \cap \HH_2$ and completes the proof.
\end{proof}

\Cref{prop: enriched Stallings II} extends \Cref{thm: enriched Stallings computable} by describing Stallings automata of general (not necessarily finitely generated) intersections. Below, we prove that this new approach can also be made algorithmic, \emph{even when the intersection is not finitely generated}.

\begin{rem} \label{rem: expanded re}
Note that if the ingredients $\Atii,\Atiii$ are finite, then the vertex expansion $\expand{\Atii}{\Atiii,\!\Treei}$ is finite and algorithmically constructible. Furthermore, if $\Atiii$ (and so $\Treei$) is finite and $\Atii$ is recursively constructible then the vertex expansion $\expand{\Atii}{\Atiii,\!\Treei}$ is also recursively constructible.
\end{rem}

\begin{thm} \label{thm: enriched Stallings r.e.}
There exists an algorithm that, given finite subsets $\SSS_1,\SSS_2 \subseteq \Fta$,  recursively constructs a Stallings automaton for the intersection $\gen{\SSS_1} \cap \gen{\SSS_2}$.
\end{thm}

\begin{proof}
Compute (finite) Stallings automata $\Eti_{\!i}$ for $\HH_i = \gen{\SSS_i}$, $i=1,2$, and use linear algebra to compute the intersection $L_1 \cap L_2 = \HH_1 \cap \HH_2 \cap \AA $ of the respective basepoint subgroups (to obtain the basepoint subgroup of the desired automaton $\Eti$).

To recursively construct the body of $\Eti$, start computing the core $\Atiii$ of the (finite, doubly-enriched) product $\Ati_1 \times \Ati_2$ normalized \wrt a chosen spanning tree $\Treei$, and the corresponding free basis $\basis_{\Treei}=\set{w_1, \ldots, w_r}$ for $\gen{\sk(\Ati_{\!1}) \times \sk(\Ati_{\!2}}) =\HH_1 \prF \cap \HH_2 \prF$.

Once we have a free basis for $\HH_1 \prF \cap \HH_2 \prF$ we can compute
the parameters $\diagi_1,\ldots,\diagi_r$ and the multiset $\mset{\vect{e_i}\matr{Q}}_i$ from \Cref{thm: Stallings = Cayley}. Let $\Atii = \Cayley{\textstyle{\bigoplus_{i=1}^{r} \ZZ/\diagi_i \ZZ} \,, \mset{\vect{e}_i \matr{Q}}_i }$, which may be infinite (if and only if $\diagi_r = 0$) but is always recursively constructible. Indeed, (for $n=0$) let $\Atii_0$ be the subautomaton induced by the basepoint of $\Delta$ (which may include loops), and
for $n = 1,2,\ldots$ construct the $n$-th ball $\Atii_n$ by adding to $\Atii_{n-1}$ the (finitely many) vertices at distance $n$ from $\bp$,
and (by inspection) all the arcs in $\Atii$ within $\Atii_n$. (For later use, note that all the arcs added in this step have one end at distance $n$ and the other at distance either $n$ or $n-1$ from $\bp$; so any
$\bp$-walk created during this step must have length at least $2n$.)

Hereinafter, we reinterpret $\Atii = \stallings{(\HH_1 \cap \HH_2)\prF, \basis_{\Treei}}$ using the explicit bijection between the multiset $\mset{\vect{e_i}\matr{Q}}_i$ and the free basis $\basis_{\Treei}$ given in the proof of \Cref{thm: Stallings = Cayley}.

From \Cref{rem: expanded re}, $\expand{\Atii}{\Atiii,\Treei}$ is also recursively constructible. In fact, since $\Atii_{n-1}$ is a full subautomaton of $\Atii_{n}$, then $\expand{\Atii_{n-1}}{\Atiii,\Treei}$ is also a full subautomaton of $\expand{\Atii_{n}}{\Atiii,\Treei}$, which is computable by just exploding to $\Treei$ the new vertices, and adding arcs accordingly. Note that every $\expand{\Atii_n}{\Atiii,\Treei}$ is a full subautomaton of the equalizable automaton $\expand{\Atii}{\Atiii,\Treei}$, and therefore it is equalizable as well.

Finally, we extend the procedure to output a sequence $\Eti_{\!0}, \Eti_{\!1}, \Eti_{\!2}, \ldots$ recursively constructing a Stallings automaton $\Eti$ for the intersection $\HH_1\cap \HH_2$; namely, $\expand{\Atii}{\Atiii,\Treei}$
equalized \wrt to some (possibly infinite) spanning tree $\TT$: at step $n=0$ declare $\TT_{0} \coloneqq \Treei$ to be the spanning tree for $\expand{\Atii_{0}}{\Atiii,\Treei}$, and equalize \wrt to it (see \Cref{prop: enriched Stallings II}) to obtain $\Eti_{\!0}$; at step $n$ construct $\expand{\Atii_{n}}{\Atiii,\Treei}$ from $\expand{\Atii_{n-1}}{\Atiii,\Treei}$, enlarge $\TT_{n-1}$ to a spanning tree $\TT_{n}$ of $\expand{\Atii_{n}}{\Atiii,\Treei}$, and equalize the new arcs to obtain $\Eti_{\!n}$.
If we call $\TT$ the direct limit of $\set{\,\TT_{\!n} \st n \in \NN \,}$, then it is straightforward to see that $\TT$ is a spanning tree for $\Eti$, and that $\Eti_{\!0}, \Eti_{\!1}, \Eti_{\!2}, \ldots$ is a strictly increasing sequence of full subautomata of $\Eti$ whose direct limit is $\Eti$. The claimed result follows.
\end{proof}

Note that this last result immediately provides a recursive enumeration of a basis for the intersection $\HH_1 \cap \HH_2$.
Furthermore, since the enumeration 
can be made in increasing order (\eg \wrt the word length of the free parts), it turns out that we can obtain a recursive basis.

\begin{cor} \label{cor: recursive bases for the intersection}
Let $\HH_1,\HH_2 
$ be two finitely generated subgroups of $\Fta$ given by finite sets of generators. Then $\HH_1\cap \HH_2$ has a recursive basis, which can be effectively computed.
\end{cor}

\begin{proof}
Let $\Eti$ be the Stallings automaton for $\HH_1 \cap \HH_2$ recursively described in the proof of \Cref{thm: enriched Stallings r.e.}. Since the basis for the abelian part is always finite,
it is enough to see that a recursive basis $\Basis_{\TT}$ (of the free part of the intersection described by the body of~$\Eti$) can be obtained. Following the notation in the previous proof, let $\Basis_n$ denote the enriched $\TT_{\!n}$-basis of $\Eti_{\!n}$ (which is obviously computable since $\Eti_{\!n}$ is finite). Then, it is clear that the increasing sequence $\Basis_{0},\Basis_{1},\Basis_{2},\ldots$ entails a recursive enumeration of the $\TT$-basis
$\Basis_{\TT} = \bigcup_{n \in \NN} \Basis_n$. Finally, note that every $\bp$-walk in $\Atii$ passing trough an arc outside $\Atii_{n-1}$ has length at least $2n$. Since vertex expansions do not decrease the length of petals, the same is true (after expanding) for the $\TT$-petals of $\Ati$ not included in $\Ati_{\!n-1}$. Therefore, the free parts of the elements in $\Basis_{\TT} \setmin \Basis_{n-1}$ are all of length at least $2n$. Now the decision of membership for $\Basis_{\TT}$ is straightforward: given a candidate element $\fta{u}{a} \in \FTA$ with $\lambda=|u|$, it is enough to check whether it belongs to the finite portion $\Basis_{\ceil{(\lambda - 1)/2}}$ of $\Basis_{\TT}$; if so, answer \yep, and otherwise answer \nop\ (since the rest of elements in $\Basis_{\TT}$ have length at least $2(\ceil{(\lambda - 1)/2} + 1) >\lambda$). Hence, $\Basis_{\TT}$ is recursive, and the proof is complete.
\end{proof}

\section{Applications to the index of subgroups} \label{sec: index}

For a general group $G$ and a subgroup $H\leqslant G=\gen{X}$, the Schreier graph $\schreier{H,X}$ has as vertices the set of (right) cosets
of $G$ modulo $H$; so, knowing $\schreier{H,X}$ we can determine a set of coset representatives for $H\leqslant G$, and decide if the subgroup has finite or infinite index. This is the case in the free group: for a finitely generated subgroup $H\leqslant \Free[X]$, one can compute $\stallings{H,X}=\core(\schreier{H,X})$ and decide whether $H$ is of finite index by checking whether $\stallings{H,X}$ is \defin{saturated} (\ie every vertex is the origin of an $x$-arc, for every $x \in X^{\pm}$); in this case $H$ is of finite index and the labels of selected paths from the basepoint $\bp$ to each vertex $\verti$ in $\stallings{H,X}$ (for example, through a chosen spanning tree~$\Treei$) form a finite transversal; otherwise, $H$ is of infinite index and we can recursively enumerate a transversal by constructing and reading bigger and bigger portions of all the hanging trees in $\schreier{H,X}$ going out of $\stallings{H,X}$. Furthermore, since this enumeration can be made in increasing order of the length of the elements, the obtained transversal is a recursive subset of $\Free[X]=\Fn$.

We aim to use our enriched Stallings machinery to
understand
the
index of a  subgroup~$\HH \leqslant \FTA
$ given by a basis $\set{\fta{u_1}{a_1},\ldots ,\fta{u_p}{a_p}, \fta{}{b_1},\ldots ,\fta{}{b_q}}$, where $\gen{\vect{b_1},\ldots ,\vect{b_q}} = L_{\HH} = \HH \cap \AA$.
Applying well-known general properties of the index of intersections and direct products we have:
\begin{equation} \label{eq: index of HH}
\max \big\{
\ind{ \Fn }{ \HH \cap \Fn } ,
\ind{ \AA }{ L_{\HH}}
\big\}
\,\leq\,
\Ind{ \FTA }{ \HH }
\,\leq\,
 \ind{ \Fn }{ \HH \cap \Fn }
 \cdot
 \ind{ \AA }{ L_{\HH} } \,.
\end{equation}
Since
$\ind{ \Fn }{ \HH \cap \Fn }
=
\ind{ \Fn }{ \HH \prF} \cdot
\ind{ \HH \prF }{ \HH \cap \Fn }$,
the index
$\Ind{\FTA}{\HH}$ is finite if and only if all three indices
$\ind{ \Fn }{ \HH \prF }$,
$\ind{ \HH \prF }{ \HH \cap \Fn }$, and
$\ind{\AA }{ L_{\HH} }$ are finite.
Furthermore, the index
$\ind{\HH \prF}{\HH \cap \Fn}$
is the number of vertices in
$\schreier{\HH \cap \Fn, \basis}$, where $\basis$ is a basis of $\HH\prF = \HH \prF \cap \Fn $.
Taking $\HH_1 = \HH$ and $\HH_2 = \Fn$ in \Cref{fig: intersection diagram big}
we have $r=p$, $\matr{B_1} = \Matrid $, $\matr{A_1}$ has $\vect{a_i}$ as its $i$-th row ($i=1,\ldots ,p$), and $\matr{A_2} = \matr{0}$; hence,
$\matr{D} = \matr{A_1}$,
and from \eqref{eq: chain of isomorphisms of fg quotient groups},
$
\ind{\HH \prF}{\HH \cap \Fn}
= \ind{\ZZ^{r}}{(L_{\HH})\matr{A_1}^{-1}}
\leq
\ind{\AA}{L_{\HH}}$.
In particular, if $\ind{\AA}{L_{\HH}} < \infty$, then $\ind{\HH \prF}{\HH \cap \Fn} < \infty$ and therefore
\begin{equation*}
    \Ind{\FTA}{\HH} < \infty
    \,\Biimp\,
    \ind{\Fn}{\HH \prF} < \infty
    \text{\, and \,}
    \ind{\AA}{L_{\HH}} < \infty \,.
\end{equation*}

Furthermore, the fact that the quotient $\HH \prF / (\HH \cap \Fn)$ does not contribute to the finiteness of the index $\Ind{\FTA}{\HH}$ suggests the possibility that it might indeed not contribute to the index at all, which turns out to be true and straightforward to prove.

\begin{prop} \label{prop: index of HH}
Let $\HH$ be a subgroup of $\FTA
$, let
$\set{v_i}_{i\in I}$ be a right transversal for $\HH \prF$ in $\Fn$,
and let
$\set{\vect{c_j}}_{j\in J}$ be a transversal for $L_{\HH} = \HH \cap \AA$ in $\AA$. Then,
$
    \Set{ \fta{v_i}{c_j} \st i \in I, j \in J}
$ 
is a right transversal for $\HH$ in $\FTA$. Hence,
$\Ind{\FTA}{\HH} = \ind{\Fn}{\HH \prF} \cdot \ind{\AA}{L_{\HH}}$; in particular,
the index
$\Ind{\FTA}{\HH}$ is finite
if and only if both
$\ind{\Fn}{\HH \prF}$
and
$\ind{\AA}{L_{\HH}}$ are finite.
\end{prop}

\begin{proof}
Let $\Fn = \bigsqcup_{i \in I}  (\HH\prF) v_i$ and $\AA = \bigsqcup_{j \in J} (L_{\HH} + \vect{c_j})$.
We first claim that the elements in $\set{\, \fta{v_i}{c_j} \st i \in I, j \in J \,}$ are all different from each other modulo $\HH$.
Indeed, if
$\HH \fta{v_i}{c_j} =  \HH \fta{v_{i'}}{c_{j'}}$ then, projecting to $\Fn$, $(\HH\prF) v_i = (\HH\prF) v_{i'}$ and $i=i'$. Hence, $\HH \fta{}{c_j} = \HH \fta{}{c_{j'}}$. Now, intersecting with $\AA$, we obtain $L_{\HH} + \vect{c_j} = L_{\HH} + \vect{c_{j'}}$ and so $j=j'$. On the other hand, we claim that
\smash{$\sqcup_{i\in I} \sqcup_{j\in J} \HH \fta{v_i}{c_j} = \Fn \times \AA$}. In fact, for an arbitrary element $\fta{w}{a}\in \Fn\times \AA$, we have $w=uv_i$ for some $i\in I$ and $u\in \HH\prF$; choose $\vect{b}\in \AA$ so that $\fta{u}{b}\in \HH$ and write $\vect{a}-\vect{b}=\vect{l}+\vect{c_j}$ for some $j\in J$ and $\vect{l}\in L_{\HH}$; then, $\fta{w}{a}=\fta{uv_i}{a}=(\fta{u}{b}\cdot \fta{}{l})\cdot \fta{v_i}{c_j} \in \HH \fta{v_i}{c_j}$. This completes the proof.
\end{proof}

So, 
a system of coset representatives (and hence the index) of a subgroup $\HH \leqslant \FTA$ is transparently encoded in any enriched Stallings automaton $\Eti$ for $\HH$.
In particular, $\HH$ is of finite index in $\FTA$ if and only if the basepoint subgroup
of $\Ati$ is of finite index in $\AA$ and $\sk(\Ati)$ is saturated. Moreover, since Stallings automata for finitely generated subgroups are computable (\Cref{thm: enriched Stallings computable}), all this information is available algorithmically, and one can effectively decide whether the index $\Ind{\FTA}{\HH}$ is finite.

Furthermore, when the index is infinite (and $\HH$ is finitely generated),  a transversal for $L_{\HH}$ in $\AA$ is recursively enumerable using basic linear algebra techniques, and a transversal for $\HH\prF$ in $\Fn$ is also recursively enumerable (by reading first the finite core $\sk (\Eti)$ and then bigger and bigger portions of all the hanging trees in the Schreier graph $\schreier{\HH\prF, X}$ going out of $\sk (\Eti)$). According to~\Cref{prop: index of HH}, combining these two recursive enumerations we can recursively enumerate a transversal for $\HH$. Moreover, since these two recursive enumerations can be done in increasing order (say, of the sum of absolute values of the coordinates, and of the word length, respectively) the obtained transversal  is indeed  recursive. The last claims are summarized below.

\begin{prop}
Let $\HH\leq \FTA$ be a finitely generated subgroup given by a finite set of generators. Then, (i) there is an algorithm to decide whether $\HH$ is of finite index and, in the affirmative case, compute the index and a transversal for $\HH$ (\ie $\FIP(\FTA)$ is solvable); and (ii) $\HH$ has a recursive transversal, which can be effectively computed. \qed
\end{prop}

\begin{rem}
Note that our geometric argument improves the proof for $\FIP(\FTA)$ given in \cite[]{DelgadoAlgorithmicProblemsFreeabelian2013} by removing all the possible redundancy in the coset description, making unnecessary the (computationally expensive) cleaning procedure used there.
\end{rem}

One last straightforward application of~\Cref{prop: index of HH} is the extension of \citeauthor{HallSubgroupsFiniteIndex1949}'s Theorem (see~\cite{HallSubgroupsFiniteIndex1949}) to the free-times-abelian context.
A subgroup $\HH$ of a free-times-abelian group $\FTA$ is called a \defin{factor} (of $\FTA$) if some (and hence, every) basis of $\HH$ can be extended to a basis of $\FTA$ (which is equivalent to saying that $\HH\prF$ is a free factor of the free part of $\FTA$, and $L_{\HH}$ is a direct summand of the abelian part of $\FTA$);
see \cite{RoyFixedSubgroupsComputation2019}.


\begin{prop}
Every finitely generated subgroup $\HH\leqslant \Fn\times \AA$ is a factor of a finite index subgroup $\KK\leqslant \Fn\times \AA$.
\end{prop}

\begin{proof}
Let $\Ati$ be a Stallings graph for $\HH$. Add the necessary $x$-arcs, $x\in X$ (with zero abelian labels) in order to obtain a saturated automaton, and complement the basepoint subgroup $L_{\HH}$ to a finite index subgroup $L$ of $\AA$, \ie $L_{\HH}\leqslant_{\oplus} L\leqslant_{\fin} \AA$. By~\Cref{prop: index of HH}, the enriched automaton $\Ati'$ obtained in this way corresponds to a subgroup $\KK=\gen{\Ati'}$ of finite index in $\Fn\times \AA$ and, by construction, $\HH$ is a factor of $\KK$.
\end{proof}

\section{Examples} \label{sec: examples}
In this section we use enriched automata to study a couple of examples showing relevant situations that can occur when intersecting two finitely generated subgroups of $\Fta$. Recall that in the graphical representation we shall omit all the trivial abelian labels, including the basepoint subgroup.

\subsection{Moldavanski's example} \label{ssec: Moldavanski's example}

Let $\HH_1 = \gen{xt,y}$ and $\HH_2 = \gen{x,y}$ be subgroups of the  group $\Free[2] \times \ZZ = \pres{x,y}{-} \times \pres{t}{-}$. Then,
$ L_{1} = \HH_1 \cap \ZZ = L_{2} = \HH_2 \cap \ZZ = L_{1} \cap L_{2} = L_{1} + L_{2} = \set{0}$, and respective enriched Stallings automata for $\HH_1$ and $\HH_2$ are:
\vspace{-10pt}
\begin{center}
\begin{tikzpicture}[shorten >=1pt, node distance=.3cm and 2cm, on grid,auto,>=stealth']
  \node[state, accepting] (0) {};
  \node[] (c) [right = 2.5 and 2 of 0] {and };
  \node[] (S) [left = 2.6 of 0] {$\stallings{\HH_1, \set{x,y}}  \equiv $};
  \path[->]
        (0) edge[blue,loop,min distance=15mm, out=45,in=-45]
            node[pos=0.5,right] {$y$}
            (0);

  \path[->]
        (0) edge[red,loop,min distance=15mm, out=225,in=135]
            node[pos=0.5,left] {$x$}
            node[pos=0.93,above right=-0.05,black] {$\scriptstyle{1}$}
            (0);
\end{tikzpicture}
 \quad
 \begin{tikzpicture}[shorten >=1pt, node distance=.3cm and 2cm, on grid,auto,>=stealth']
  \node[state, accepting] (0) {};
  \node[] (S) [left = 2.4 of 0] {$\stallings{\HH_2, \set{x,y}}  \equiv $};
  \path[->]
        (0) edge[blue,loop,min distance=15mm, out=45,in=-45]
            (0);

  \path[->]
        (0) edge[red,loop,min distance=15mm, out=225,in=135]
            (0);
 \end{tikzpicture}
\vspace{-20pt}
\end{center}
and therefore,
\vspace{-20pt}
\begin{figure}[H]
\centering
\begin{tikzpicture}[anchor=base, baseline=-4pt,shorten >=1pt, node distance=.3cm and 2cm, on grid,auto,>=stealth']
  \node[state, accepting] (0) {};
  \node[] (S) [left = 3.75 of 0] {$ \stallings{\HH_1, \set{x,y}} \times \stallings{\HH_2, \set{x,y}}  \,\equiv\, $};
  \path[->]
        (0) edge[blue,loop,min distance=15mm, out=45,in=-45]
            (0);

  \path[->]
        (0) edge[red,loop,min distance=15mm, out=225,in=135]
            node[pos=0.9,above right=-0.1] {$\scriptstyle{(1,0)}$}
            (0);
\end{tikzpicture}
\vspace{-20pt}
\end{figure}
Note that the basis $\set{w_1,w_2}$ obtained for $\HH_1 \prF \cap \HH_2 \prF$ is exactly the same as the original basis for $\HH_1\prF$ and for $\HH_2\prF$, namely $w_1 = x$ and $w_2 = y$.
According to our scheme, $\matr{D}
=
\matr{B_1} \matr{A_1} - \matr{B_2} \matr{A_2}
=
\left(
\begin{smallmatrix}
1 & 0\\0 & 1
\end{smallmatrix}
\right)
\left(
\begin{smallmatrix}
1\\0
\end{smallmatrix}
\right)
-
\left(
\begin{smallmatrix}
1 & 0\\0 & 1
\end{smallmatrix}
\right)
\left(
\begin{smallmatrix}
0\\0
\end{smallmatrix}
\right)
=
\left(
\begin{smallmatrix}
1\\0
\end{smallmatrix}
\right)
$,
and the matrix $\matr{M} = ( 0\ 1 )$ has as row a basis for $(L_{1} + L_{2})\matr{D}^{-1} =\ker \matr{D}$. Hence, the Smith normal form of $\matr{M}$ is $\matr{S} = \matr{P} \matr{M} \matr{Q} = ( 1 \ 0)$, with~$\matr{P} = (1)$, and $\matr{Q}
=
\left(
\begin{smallmatrix}
0 & 1 \\
1 & 0
\end{smallmatrix}
\right)
$.
Therefore, $\diagi_1 = 1$, $\diagi_2 = 0$,
and applying~\Cref{thm: Stallings = Cayley} we have that
 \begin{equation*}
  \Stallings{(\HH_1 \cap \HH_2) \prF , \set{w_1,w_2}}
  \,\isom\,
  \Cayley{ \ZZ/\ZZ \oplus  \ZZ/ 0\ZZ \,, \set{(0,1),(1,0)}}
  \,\isom\,
  \Cayley{ \ZZ \,, \set{1,0}} ,
 \end{equation*}
 \begin{figure}[H]
 \centering
  \begin{tikzpicture}[shorten >=1pt, node distance=1.2 and 1.2, on grid,auto,>=stealth']
  \node[state,accepting] (0) {};
  \node[state] (1) [right = of 0]{};
  \node[state] (2) [right = of 1]{};
  \node[state] (3) [right = of 2]{};
  \node[] (4) [right = 1.5 of 3]{$\cdots$};

  \node[state] (-1) [left = of 0]{};
  \node[state] (-2) [left = of -1]{};
  \node[state] (-3) [left = of -2]{};
  \node[] (-4) [left = 1.5 of -3]{$\cdots$};

  \path[->]
        (0) edge[loop above,Green,min distance=10mm,in=55,out=125]
            node[] {\scriptsize{$w_2$}}
            (0)
            edge[violet]
            node[below] {\scriptsize{$w_1$}}
            (1);

    \path[->]
        (1) edge[loop above,Green,min distance=10mm,in=55,out=125]
            (1)
            edge[violet]
            (2);

    \path[->]
        (2) edge[loop above,Green,min distance=10mm,in=55,out=125]
            (2)
            edge[violet]
            (3);

    \path[->]
        (3) edge[loop above,Green,min distance=10mm,in=55,out=125]
            (3)
            edge[violet]
            (4);

    \path[->]
        (-1) edge[loop above,Green,min distance=10mm,in=55,out=125]
            (-1)
            edge[violet]
            (0);

    \path[->]
        (-2) edge[loop above,Green,min distance=10mm,in=55,out=125]
            (-2)
            edge[violet]
            (-1);

    \path[->]
        (-3) edge[loop above,Green,min distance=10mm,in=55,out=125]
            (-3)
            edge[violet]
            (-2);

    \path[->]
        (-4)
        edge[violet]
            (-3);
\end{tikzpicture}
\end{figure}
Since the obtained abelian group $\ZZ$ is infinite, the intersection $\HH_1 \cap \HH_2$ is not finitely generated.
Now, replace the arcs labeled by $w_1,w_2$ by the corresponding enriched paths reading $xt$ and $y$ and note that there are no foldings available. Finally, (normalize and) equalize \wrt the only possible spanning tree $\Treei$ (consisting of the $x$-labeled (red) arcs in \Cref{fig: normalcl(y)}) to obtain a Stallings automaton for $\HH_1 \cap \HH_2$.

\begin{figure}[H]
 \centering
  \begin{tikzpicture}[shorten >=1pt, node distance=1.2 and 1.2, on grid,auto,>=stealth']
  \node[state,accepting] (0) {};
  \node[state] (1) [right = of 0]{};
  \node[state] (2) [right = of 1]{};
  \node[state] (3) [right = of 2]{};
  \node[] (4) [right = 1.5 of 3]{$\cdots$};
  \node[] (c) [right = .5 of 4]{;};

  \node[state] (-1) [left = of 0]{};
  \node[state] (-2) [left = of -1]{};
  \node[state] (-3) [left = of -2]{};
  \node[] (-4) [left = 1.5 of -3]{$\cdots$};

  \path[->]
        (0) edge[loop above,blue,min distance=10mm,in=55,out=125]
            node[] {\scriptsize{$y$}}
            (0)
            edge[red]
            node[below] {\scriptsize{$x$}}
            (1);

    \path[->]
        (1) edge[loop above,blue,min distance=10mm,in=55,out=125]
            (1)
            edge[red]
            (2);

    \path[->]
        (2) edge[loop above,blue,min distance=10mm,in=55,out=125]
            (2)
            edge[red]
            (3);

    \path[->]
        (3) edge[loop above,blue,min distance=10mm,in=55,out=125]
            (3)
            edge[red]
            (4);

    \path[->]
        (-1) edge[loop above,blue,min distance=10mm,in=55,out=125]
            (-1)
            edge[red]
            (0);

    \path[->]
        (-2) edge[loop above,blue,min distance=10mm,in=55,out=125]
            (-2)
            edge[red]
            (-1);

    \path[->]
        (-3) edge[loop above,blue,min distance=10mm,in=55,out=125]
            (-3)
            edge[red]
            (-2);

    \path[->]
        (-4)
        edge[red]
            (-3);
\end{tikzpicture}
\caption{Stallings automaton for $\gen{xt,y} \cap \gen{x,y}$}
\label{fig: normalcl(y)}
\end{figure}
Therefore $\HH_1 \cap \HH_2$ is not finitely generated, and $\Basis_{\Treei} = \set{\,x^i y x^{-i} \st i\in \ZZ\,}$ is a basis for $  \HH_1 \cap \HH_2 = \normalcl{y}$.

\subsection{Parameterized example} \label{ssec: parameterized example}
Consider the subgroups
$\HH_1 \,=\, \gen{\,  \fta{x^3}{a},\fta{yx}{b}, \fta{y^3 x y^{-2}}{c}, \T^{L_1}}$,
and
$\HH_2 \,=\, \gen{\,\fta{x^2}{d},yxy^{-1}, \T^{{L_2}}\,}$
of the direct product~${\Free[2] \times \ZZ^2}$,
where $\vect{a},\vect{b},\vect{c},\vect{d} \in \ZZ^2$, and $L_1,L_2$ are subgroups of $\ZZ^2$.

According to our previous discussion, in order to compute  (a Stallings automaton for) the intersection $\HH_1 \cap \HH_2$ we first compute respective Stallings automata $\Eti_{\!1},\Eti_{\!2}$ for $\HH_1$ and $\HH_2$, and then build its product $\Eti_{\!1} \times \Eti_{\!2}$;
see \Cref{fig: enriched product example}.
\begin{figure}[H]
\centering
\begin{tikzpicture}[shorten >=1pt, node distance=1.2cm and 2cm, on grid,auto,auto,>=stealth']

  \node[state,accepting] (0)  {};
  \node[state] (1) [left = 2.3 of 0] {};
  \node[state] (2) [right = of 0] {};
  \node[] (L') [above right = 3mm of 0] {$\scriptstyle{L_2}$};
  \node[] (S') [above left = 5mm and 26mm of 0] {$\Ati_{\!2}$};


  \node[state,accepting] (0') [below left = of 1]  {};
  \node[state] (1') [below = of 0'] {};
  \node[state] (2') [below = of 1'] {};
  \node[state] (3') [below = of 2'] {};
  \node[state] (4') [below = of 3'] {};
  \node[] (L) [above = 3mm of 0'] {$\scriptstyle{L_1}$};
  \node[] (S') [above left = 3mm and 7mm of 0'] {$\Ati_{\!1}$};


  \path[->]
        (0) edge[red,bend right=50]
            (1);
  \path[->]
        (1) edge[red,bend right=45]
            node[pos=0.8,above] {$\vect{d}$}
            (0);


    \path[->]
        (0) edge[blue]
            (2);

    \path[->]
        (2) edge[red,loop right,min distance=15mm, out=-45,in=45]
            (2);

%
%
%
%
%
%

    \path[->]
        (0') edge[red]
            (1');


    \path[->]
        (1') edge[red]
            node[pos=0.72,right=-0.05] {$\vect{a}$}
            (2');


    \path[->]
        (2') edge[red, bend right = 60]
            (0');

    \path[->]
        (0') edge[blue, bend right = 60]
            node[pos=0.9,below left=-0.05] {$\vect{b}$}
            (2');

    \path[->]
        (2') edge[blue]
            (3');
    \path[->]
        (3') edge[blue, bend right = 45]
             node[pos=0.85,left] {$\vect{c}$}
            (4');

    \path[->]
        (4') edge[red, bend right = 45]
            (3');

  \node[state,accepting] (0'0) [below of = 0] {};
  \node[state] (0'1) [below of = 1] {};
  \node[state] (0'2) [below of = 2] {};

  \node[state] (1'0) [below of = 0'0] {};
  \node[state] (1'1) [below of = 0'1] {};
  \node[state] (1'2) [below of = 0'2] {};

  \node[state] (2'0) [below of = 1'0] {};
  \node[state] (2'1) [below of = 1'1] {};
  \node[state] (2'2) [below of = 1'2] {};

  \node[state] (3'0) [below of = 2'0] {};
  \node[state] (3'1) [below of = 2'1] {};
  \node[state] (3'2) [below of = 2'2] {};

  \node[state] (4'0) [below of = 3'0] {};
  \node[state] (4'1) [below of = 3'1] {};
  \node[state] (4'2) [below of = 3'2] {};
  \node[] (L) [above right = 2mm  and 4mm of 0'0] {$\scriptstyle{L_1,L_2}$};
  \node[] (S') [above right = 0mm and 9mm of 4'2] {$\Ati_{\!1} \times \Ati_{\!2}$};

    \path[->]
        (0'0) edge[red]
            (1'1);

    \path[->]
        (1'1) edge[red]
            node[pos=0.9,left] {$\vect{a},\!\vect{d}$}
            (2'0);

    \path[->]
        (2'0) edge[red]
            (0'1);

    \path[->]
        (0'1) edge[red]
            node[pos=0.8,right=0.05] {$\vect{0},\!\vect{d}$}
            (1'0);

    \path[->]
        (1'0) edge[red]
            node[pos=0.9,right=0.1] {$\vect{a},\!\vect{0}$}
            (2'1);

    \path[->]
        (2'1) edge[red]
            node[pos=0.85,right] {$\vect{0},\!\vect{d}$}
            (0'0);

    \path[->]
        (0'0) edge[blue]
            node[pos=0.92,left=0.03] {$\vect{b},\!\vect{0}$}
            (2'2);

    \path[->]
        (2'2) edge[red,bend right = 75]
            (0'2);

    \path[->]
        (0'2) edge[red]
            (1'2);

    \path[->]
        (1'2) edge[red]
            node[pos=0.75,right=-0.05] {$\vect{a},\!\vect{0}$}
            (2'2);
    \path[->]
        (4'2) edge[red]
            (3'2);

    \path[->]
        (4'1) edge[red]
            node[pos=0.9,above left = -0.1] {$\vect{0},\!\vect{d}$}
            (3'0);

    \path[->]
        (4'0) edge[red]
            (3'1);

    \path[->]
        (2'0) edge[blue]
            (3'2);

    \path[->]
        (3'0) edge[blue]
            node[pos=0.78,above right=-0.1] {$\vect{c},\!\vect{0}$}
            (4'2);
\end{tikzpicture}
\caption{Product $\Eti_{\!1} \times \Eti_{\!2}$ of the Stalings automata $\Eti_{\!1},\Eti_{\!2}$ for $\HH_1$ and $\HH_2$}
\label{fig: enriched product example}
\end{figure}

Note that the product $\Ati_{\!1} \times \Ati_{\!2}$ is disconnected and has a hanging tree not containing the basepoint. After removal, we obtain the core of the doubly-enriched product which can be normalized as follows (the arcs outside the chosen spanning tree are drawn with thicker lines):

\begin{figure}[H]
\centering
\begin{tikzpicture}[shorten >=1pt, node distance=1.2cm and 2cm, on grid,auto,>=stealth']
  \node[state,accepting] (0'0) {};
  \node[state] (1'1) [below left of = 0'0] {};
  \node[state] (2'0) [left of = 1'1] {};
  \node[state] (0'1) [above left  of = 2'0] {};
  \node[state] (1'0) [above right of = 0'1] {};
  \node[state] (2'1) [right of = 1'0] {};

  \node[state] (2'2) [right = 1.5 of 0'0] {};
  \node[state] (0'2) [below right of = 2'2] {};
  \node[state] (1'2) [above right of = 2'2] {};
  \node[] (bp) [above right = 0.3 and 0.3 of 0'0] {$\scriptstyle{L_1,L_2}$};


    \path[->]
        (1'1) edge[red]
            (0'0);

    \path[->]
        (1'0) edge[red]
            (0'1);

    \path[->]
        (2'1) edge[red]
            (1'0);

    \path[->]
        (2'0) edge[red]
            (1'1);

    \path[->]
        (0'1) edge[thick,red]
             node[pos=0.75,below left=-0.1] {$2\vect{a},\!3\vect{d}$}
            (2'0);

    \path[->]
         (0'0) edge[red]
            (2'1);

    \path[->]
        (0'0) edge[blue]
            (2'2);

    \path[->]
        (0'2) edge[thick,red]
             node[pos=0.8,right] {$\vect{a},\!\vect{0}$}
            (1'2);

    \path[->]
        (1'2) edge[red]
            (2'2);

    \path[->]
        (2'2) edge[red]
            (0'2);
\end{tikzpicture}
\caption{Normalized product for $\HH_1 \cap \HH_2$}
\label{fig: normalized intersection scheme example}
\end{figure}

\begin{rem} \label{rem: neglected parts of product}
The basis element $y^3 x y^{-2}\, \fta{}{c} \in \HH_1$ does not contribute to the core of the product $\Eti_{\!1} \times \Eti_{\!2}$.
In a similar vein,
the abelian labels $\vect{b},\vect{c}$
no longer appear in the normalized product core (\Cref{fig: normalized intersection scheme example})
and will not play any role in the intersection $\HH_1 \cap \HH_2$.
\end{rem}
So,
we obtain a basis
$\set{ w_1,w_2}$  for $\HH_1 \prF \cap \HH_2 \prF $, where $w_1 = x^6$ and $w_2 = y x^3 y^{-1}$.
Let us now study the intersection $\HH_1 \cap \HH_2$ in light of
\Cref{thm: Stallings = Cayley,thm: enriched Stallings computable}.
According to the notation summarized in \Cref{fig: intersection diagram big}, we have
$\matr{A_1} =
\left(
  \begin{smallmatrix}
  \vect{a}  \phantom{\vect{d}} \hspace{-5pt}\\
  \vect{b}\\
  \vect{c}
  \end{smallmatrix}
  \right)
$,
$\matr{A_2} =
\left(
  \begin{smallmatrix}
  \vect{d} \\
  \vect{0}\\
  \end{smallmatrix}
  \right)
$,
$\matr{B_1} =
\left(
  \begin{smallmatrix}
  2&0 & 0\\ 1&0&0\\
  \end{smallmatrix}
  \right)
$,
$\matr{B_2} =
\left(
  \begin{smallmatrix}
  3&0\\ 0&3\\
  \end{smallmatrix}
  \right)
$,
and
$\matr{D} =
\left(
  \begin{smallmatrix}
  2\vect{a}  - 3 \vect{d}\\ \vect{a}\\
  \end{smallmatrix}
  \right)
$.
We shall distinguish different cases depending on the values of the parameters $\vect{a},\vect{d} \in \ZZ^2$, and the subgroups~$L_1,L_2\,\leqslant \ZZ^2 $.

\goodbreak

\paragraph{\large{Case 1.}}
Let $\vect{a} = (1,0),\,\vect{d}=(0,1) \in \ZZ^2$, and $L_{1} = \gen{(0,6)}$, $L_{2}=\gen{(3,-3)} \leqslant \ZZ^2$.

Then, $L_1 + L_2 = \gen{(0,6),(3,-3)}$, $L_{1} \cap L_{2} = \set{(0,0)}$,
 and
 $
  \matr{D} =
  \left(
  \begin{smallmatrix}
  2 & -3\\1&\hfill 0
     \end{smallmatrix}
     \right)
$.
Hence, the subgroup ${M = (L_1 + L_2) \matr{D}\preim}$ is
      generated by the rows of the matrix~$
      {\matr{M} =
      \left(
      \begin{smallmatrix*}[r]
      -2 & 4\\1& 1
      \end{smallmatrix*}
      \right)}
      $
      which,
      in turn,
      admits the Smith normal form decomposition
      $\matr{P} \matr{M} \matr{Q} = \matr{S}$,
 where $\matr{P} =
  \left(
  \begin{smallmatrix*}[r]
    0 & 1 \\
    1 & 2
  \end{smallmatrix*}
  \right)
  $,
  $\matr{Q} =
  \left(
  \begin{smallmatrix*}[r]
    1 & -1 \\
    0 & 1
  \end{smallmatrix*}
  \right)
  $,
  and
  $\matr{S} =
  \left(
  \begin{smallmatrix*}[r]
    1 & 0 \\
    0 & 6
  \end{smallmatrix*}
  \right)
  $.
  Therefore, according to~\Cref{thm: Stallings = Cayley}, we obtain: %
  \begin{equation*}
  \Stallings{(\HH_1 \cap \HH_2) \prF, \set{w_1,w_2}} \isom
  \Cayley{\ZZ / \ZZ \oplus \ZZ/6\ZZ,\set{(1,-1),(0,1)}} \isom
  \Cayley{\ZZ/6\ZZ \,,\, \set{-1,1}} .
  \end{equation*}
Denoting by a violet (\resp green) arc the action of the element $-1$ (\resp $1$), we obtain:
\begin{figure}[H]
\centering
\begin{tikzpicture}[shorten >=1pt, node distance=1.2 and 1.5, on grid,auto,>=stealth']
\newcommand{\aaa}{6}
\node[regular polygon, regular polygon sides=\aaa, minimum size=2.5cm] at (\aaa*4,0) (A) {};

\foreach \i in {1,...,5}
    \node[state] (\i) at (A.corner \i) {};
\node[state,accepting] (6) at (A.corner 6) {};

  \path[->]
        (1) edge[bend right,violet]
            (2);

  \path[->]
        (2) edge[bend right,violet]
            (3);

  \path[->]
        (3) edge[bend right,violet]
            (4);

  \path[->]
        (4) edge[bend right,violet]
            (5);

  \path[->]
        (5) edge[bend right,violet]
            (6);

  \path[->]
        (6) edge[bend right,violet]
            node[above right = -0.1] {\scriptsize{$w_1$}}
            (1);

  \path[->]
        (2) edge[bend right,Green]
            (1);

  \path[->]
        (3) edge[bend right,Green]
            (2);

  \path[->]
        (4) edge[bend right,Green]
            (3);

  \path[->]
        (5) edge[bend right,Green]
            (4);

  \path[->]
        (6) edge[bend right,Green]
            (5);

  \path[->]
        (1) edge[bend right,Green]
            node[below left = -0.1] {\scriptsize{$w_2$}}
            (6);

\end{tikzpicture}
\caption{Stallings automaton corresponding to $\Cayley{\ZZ/6\ZZ \,,\, \set{-1,1}}$}
\label{fig: St intersection prf example1}
\end{figure}
Since $\ZZ / 6 \ZZ$ is finite, the intersection $\HH_1 \cap \HH_2$ is finitely generated.
Finally, we apply \Cref{thm: enriched Stallings computable} to compute a Stallings automaton.
After replacing the arcs reading
$w_1$ (\resp $w_2$) with
an enriched path reading
$x^6 \, \T^{(2,0),(0,3)}$
(\resp
$
y x^3 y^{-1}\, \T^{(1,0),(0,0)}$),
folding,
and normalizing \wrt a spanning tree $\Treei$
(whose cyclomatic arcs are drawn thicker),
the automaton in \Cref{fig: St intersection prf example1}
becomes:
\begin{figure}[H]
\centering
\begin{tikzpicture}[shorten >=0.5pt, node distance=1.2 and 1.5, on grid,auto,>=stealth']
\def \n {36}
\def \m {18}
\def \o {6}
\def \R {2.5cm}
\def \r {1.5cm}

\def \margin {2} 

\node[state,accepting] (0) at (2.5,0) {};
\node [above right = 0.15 and 0.5 of 0] {$\scriptstyle{L_1,L_2}$};

\foreach \s in {0,...,35}
{
  \node[state, inner sep=0pt, minimum size=3pt] at ({360/\n * (\s)}:\R) {};
}

\foreach \s in {0,...,35}
{
  \path[->]
        ({360/\n * (\s)+\margin*.8}:\R)
        edge[red,thin]
        ({360/\n * (\s+1)-\margin*.7}:\R);
}

\foreach \s in {0,...,17}
{
  \node[state, inner sep=0pt, minimum size=3pt] at ({360/\m * (\s - 1)}:\r) {};
}

\foreach \s in {0,...,17}
{
  \path[->]
        ({360/\m * (\s+1)-\margin*1.25}:\r)
        edge[red,thin]
        ({360/\m * (\s)+\margin}:\r);
}

\foreach \s in {1,...,\o}
{
  \path[->]
        ({360/\o * (\s - 1)}:\R-\margin*1.5)
        edge[blue,thin]
        ({360/\o * (\s - 1)}:\r+\margin);
}

 \path[->]
        ({360/\n * 5+\margin*.8}:\R)
        edge[red,line width=0.8pt]
      node[pos=0.7,above right = -0.1] {$\scriptscriptstyle{(-3,0),(0,3)}$}
        ({360/\n * 6-\margin*.7}:\R);

 \path[->]
        ({360/\n * 11+\margin*.8}:\R)
        edge[red,line width=0.8pt]
      node[pos=0.8,above left=-0.1] {$\scriptscriptstyle{(3,0),(0,3)}$}
        ({360/\n * 12-\margin*.7}:\R);

 \path[->]
        ({360/\n * 17 +\margin*.8}:\R)
        edge[red,line width=0.8pt]
      node[pos=0.7,left] {$\scriptscriptstyle{(3,0),(0,3)}$}
        ({360/\n * 18-\margin*.7}:\R);

 \path[->]
        ({360/\n * 23+\margin*.8}:\R)
        edge[red,line width=0.8pt]
      node[pos=0.8,below left=-0.1] {$\scriptscriptstyle{(3,0),(0,3)}$}
        ({360/\n * 24-\margin*.7}:\R);

 \path[->]
        ({360/\n * 29+\margin*.8}:\R)
        edge[red,line width=0.8pt]
      node[pos=0.7,below right=-0.1] {$\scriptscriptstyle{(3,0),(0,3)}$}
        ({360/\n * 30-\margin*.7}:\R);

 \path[->]
        ({360/\n * 35+\margin*.8}:\R)
        edge[red,line width=0.8pt]
      node[pos=0.7,right] {$\scriptscriptstyle{(3,0),(0,3)}$}
        ({360/\n * 36-\margin*.7}:\R);

  \path[->]
        ({360/\m * 1 - \margin*1.25}:\r)
        edge[red,line width=0.8pt]
        ({360/\m * 0 +\margin}:\r);

    \path[->]
        ({360/\m * 4 - \margin*1.25}:\r)
        edge[red,thin]
        ({360/\m * 3 +\margin}:\r);

    \path[->]
        ({360/\m * 7 - \margin*1.25}:\r)
        edge[red,thin]
        ({360/\m * 6 +\margin}:\r);

    \path[->]
        ({360/\m * 10 - \margin*1.25}:\r)
        edge[red,thin]
        ({360/\m * 9 +\margin}:\r);

    \path[->]
        ({360/\m * 13 - \margin*1.25}:\r)
        edge[red,thin]
        ({360/\m * 12 +\margin}:\r);

    \path[->]
        ({360/\m * 16 - \margin*1.25}:\r)
        edge[red,thin]
        ({360/\m * 15 +\margin}:\r);


\path[->]
        ({360/\m -\margin*1.1}:\r)
        edge[red]
        node[pos=0.8, left] {$\scriptscriptstyle{(6,0),(0,0)}$}
        ({\margin*.8}:\r );
\end{tikzpicture}
\caption{Normalized expanded product for $\HH_1 \cap \HH_2$}
\label{fig: normalized expanded intersection scheme example}
\end{figure}
We know by construction (see~\Cref{thm: Stallings = Cayley}) that
the automaton in \Cref{fig: normalized expanded intersection scheme example} must be equalizable;
that is,
the doubly enriched label $(\vect{a},\vect{b})$ of any $\Treei$-arc satisfies
$(\vect{a} + L_1) \cap (\vect{b} + L_2) \neq \varnothing$. After replacing each label $(\vect{a},\vect{b})$ with some $\vect{c} \in (\vect{a} + L_1) \cap (\vect{b} + L_2)$,
and replacing $(L_1,L_2)$ with $L_1 \cap L_2 = \set{(\vect{0},\vect{0})}$ as the subgroup basepoint,
we obtain a Stallings automaton for the intersection $\HH_1 \cap \HH_2$:
\begin{figure}[H]
\centering
\begin{tikzpicture}[shorten >=0.5pt, node distance=1.2 and 1.5, on grid,auto,>=stealth']
\def \n {36}
\def \m {18}
\def \o {6}
\def \R {2.5cm}
\def \r {1.5cm}

\def \margin {2} 

\node[state,accepting] (0) at (2.5,0) {};


\foreach \s in {0,...,35}
{
  \node[state, inner sep=0pt, minimum size=3pt] at ({360/\n * (\s)}:\R) {};
}

\foreach \s in {0,...,35}
{
  \path[->]
        ({360/\n * (\s)+\margin*.8}:\R)
        edge[red,thin]
        ({360/\n * (\s+1)-\margin*.7}:\R);
}

\foreach \s in {0,...,17}
{
  \node[state, inner sep=0pt, minimum size=3pt] at ({360/\m * (\s - 1)}:\r) {};
}

\foreach \s in {0,...,17}
{
  \path[->]
        ({360/\m * (\s+1)-\margin*1.25}:\r)
        edge[red,thin]
        ({360/\m * (\s)+\margin}:\r);
}

\foreach \s in {1,...,\o}
{
  \path[->]
        ({360/\o * (\s - 1)}:\R-\margin*1.5)
        edge[blue,thin]
        ({360/\o * (\s - 1)}:\r+\margin);
}

 \path[->]
        ({360/\n * 5+\margin*.8}:\R)
        edge[red,line width=0.8pt]
      node[pos=0.7,above right = -0.1] {$\scriptscriptstyle{(-3,6)}$}
        ({360/\n * 6-\margin*.7}:\R);

 \path[->]
        ({360/\n * 11+\margin*.8}:\R)
        edge[red,line width=0.8pt]
      node[pos=0.8,above left=-0.1] {$\scriptscriptstyle{(3,0)}$}
        ({360/\n * 12-\margin*.7}:\R);

 \path[->]
        ({360/\n * 17 +\margin*.8}:\R)
        edge[red,line width=0.8pt]
      node[pos=0.7,left] {$\scriptscriptstyle{(3,0)}$}
        ({360/\n * 18-\margin*.7}:\R);

 \path[->]
        ({360/\n * 23+\margin*.8}:\R)
        edge[red,line width=0.8pt]
      node[pos=0.8,below left=-0.1] {$\scriptscriptstyle{(3,0)}$}
        ({360/\n * 24-\margin*.7}:\R);

 \path[->]
        ({360/\n * 29+\margin*.8}:\R)
        edge[red,line width=0.8pt]
      node[pos=0.7,below right=-0.1] {$\scriptscriptstyle{(3,0)}$}
        ({360/\n * 30-\margin*.7}:\R);

 \path[->]
        ({360/\n * 35+\margin*.8}:\R)
        edge[red,line width=0.8pt]
      node[pos=0.7,right] {$\scriptscriptstyle{(3,0)}$}
        ({360/\n * 36-\margin*.7}:\R);

  \path[->]
        ({360/\m * 1 - \margin*1.25}:\r)
        edge[red,line width=0.8pt]
        ({360/\m * 0 +\margin}:\r);

    \path[->]
        ({360/\m * 4 - \margin*1.25}:\r)
        edge[red,thin]
        ({360/\m * 3 +\margin}:\r);

    \path[->]
        ({360/\m * 7 - \margin*1.25}:\r)
        edge[red,thin]
        ({360/\m * 6 +\margin}:\r);

    \path[->]
        ({360/\m * 10 - \margin*1.25}:\r)
        edge[red,thin]
        ({360/\m * 9 +\margin}:\r);

    \path[->]
        ({360/\m * 13 - \margin*1.25}:\r)
        edge[red,thin]
        ({360/\m * 12 +\margin}:\r);

    \path[->]
        ({360/\m * 16 - \margin*1.25}:\r)
        edge[red,thin]
        ({360/\m * 15 +\margin}:\r);


\path[->]
        ({360/\m -\margin*1.1}:\r)
        edge[red]
        node[pos=0.8, left] {$\scriptscriptstyle{(6,-6)}$}
        ({\margin*.8}:\r );
\end{tikzpicture}
\caption{Stallings automaton for $\HH_1 \cap \HH_2$ (Case 1)}
\label{fig: Stallings intersection example1}
\end{figure}
This provides the basis
$\set{\,
y  x^3  y^{-1}  x^{6}  \T^{(3,0)}, \allowbreak
y  x^6  y^{-1}  x^{6}  y  x^{-3}  y^{-1} \, \T^{(3,0)}, \allowbreak
y  x^9  y^{-1}  x^{6}  y  x^{-6}  y^{-1} \, \T^{(3,0)}, \allowbreak
y  x^{12}  y^{-1}  x^{6}  y  x^{-9}  y^{-1} \, \T^{(3,0)}, \allowbreak
y  x^{15}  y^{-1}  x^{6}  y  x^{-12}  y^{-1} \, \T^{(3,0)}, \allowbreak
y  x^{18}  y^{-1} \, \T^{(6,-6)}, \allowbreak
x^{6}  y  x^{-15}  y^{-1} \, \T^{(-3,6) \,}
}$
for the intersection $\HH_1 \cap \HH_2$.

\paragraph{\large{Case 2.}}

Let $\vect{a} = (3,3),\,\vect{d}=(2,2) \in \ZZ^2$, and $L_{1} = \gen{(1,2)},\,L_{2}=\gen{(0,0)} \leqslant \ZZ^2$.

Then, $L_1 + L_2 = \gen{(1,2)}$, $L_{1} \cap L_{2} = \set{(0,0)}$,
      $
      \matr{D} =
      \left(
      \begin{smallmatrix}
      0 & 0\\3&\hfill 3
      \end{smallmatrix}
      \right)
      $,
      and $M$ is
      generated by the row of the matrix~$
      {\matr{M} =
      \left( \,
        1 \ \, 0
        \, \right)}
      $,
      which is already in Smith normal form;
 hence, 
 $\matr{P} = (1), \ \matr{Q} =
  \left(
  \begin{smallmatrix*}[r]
    1 & 0 \\
    0 & 1
  \end{smallmatrix*}
  \right)$,
  and
  $ \matr{S} = (\, 1 \ 0\,)
  $.
  According to~\Cref{thm: Stallings = Cayley},
  $\Stallings{(\HH_1 \cap \HH_2) \prF, \set{w_1,w_2}} \isom
  \Cayley{\ZZ / \ZZ \oplus \ZZ / 0\ZZ, \set{(1,0),(0,1)}} \isom
  \Cayley{\ZZ \,, \set{0,1}}
  $,
  which takes the form:

\begin{figure}[H]
\centering
  \begin{tikzpicture}[shorten >=1pt, node distance=1.2 and 1.6, on grid,auto,>=stealth']
  \node[state,accepting] (0) {};
  \node[state] (1) [right = of 0]{};
  \node[state] (2) [right = of 1]{};
  \node[state] (3) [right = of 2]{};
  \node[] (dr) [right = .5 of 3]{$\cdots$};

  \node[state] (-1) [left = of 0]{};
  \node[state] (-2) [left = of -1]{};
  \node[state] (-3) [left = of -2]{};
  \node[] (dl) [left = .5 of -3]{$\cdots$};

    \path[->]
        (0) edge[violet,loop,min distance=15mm, out=135,in=45]
            node[above=0.01] {$w_1$}
            (0);

    \path[->]
        (1) edge[violet,loop,min distance=15mm, out=135,in=45]
            (1);

    \path[->]
        (2) edge[violet,loop,min distance=15mm, out=135,in=45]
            (2);

    \path[->]
        (3) edge[violet,loop,min distance=15mm, out=135,in=45]
            (3);

    \path[->]
        (-1) edge[violet,loop,min distance=15mm, out=135,in=45]
            (-1);

    \path[->]
        (-2) edge[violet,loop,min distance=15mm, out=135,in=45]
            (-2);

    \path[->]
        (-3) edge[violet,loop,min distance=15mm, out=135,in=45]
            (-3);

  \path[->]
        (0) edge[Green]
            node[pos=0.5,above] {$w_2$}
            (1);

    \path[->]
        (1) edge[Green]
            (2);

    \path[->]
        (2) edge[Green]
            (3);


    \path[->]
        (-1) edge[Green]
            (0);

    \path[->]
        (-2) edge[Green]
            (-1);

    \path[->]
        (-3) edge[Green]
            (-2);

\end{tikzpicture}
 \caption{\protect{Stallings automaton corresponding to $\Cayley{\ZZ ,\set{0,1}}$}}
 \end{figure}
Since $\ZZ$ is infinite, in Case 2 the intersection $\HH_1 \cap \HH_2$ has infinite rank.
After replacing the arcs reading
$w_1$ and $w_2$ with the enriched paths reading
$
x^6 \, \T^{(6,6),(6,6)}$ and
$
y x^3 y^{-1}\, \T^{(3,3),(0,0)}$,
folding, normalizing 
(\wrt the spanning tree having as cyclomatic arcs the thicker ones), and equalizing, we obtain a Stallings automaton for $\HH_1 \cap \HH_2$:
 \begin{figure}[H]
\centering
  \begin{tikzpicture}[shorten >=1pt, node distance=0.6 and 0.7, on grid,auto,>=stealth']
  \node[state,accepting] (0) {};
  \node[smallstate] (01) [below = of 0] {};
  \node[smallstate] (02) [right = of 01] {};
  \node[smallstate] (03) [right = of 02] {};
  \node[smallstate] (04) [right = of 03] {};

  \node[smallstate] (h02) [above left =0.6 of 0] {};
  \node[smallstate] (h03) [above = of h02] {};
  \node[smallstate] (h04) [above right = 0.6 of h03] {};
  \node[smallstate] (h05) [below right = 0.6 of h04] {};
  \node[smallstate] (h06) [below = of h05] {};

  \node[smallstate] (1) [above = of 04]{};
  \node[smallstate] (12) [right = of 04] {};
  \node[smallstate] (13) [right = of 12] {};
  \node[smallstate] (14) [right = of 13] {};
  \node[smallstate] (h12) [above left =0.6 of 1] {};
  \node[smallstate] (h13) [above = of h12] {};
  \node[smallstate] (h14) [above right = 0.6 of h13] {};
  \node[smallstate] (h15) [below right = 0.6 of h14] {};
  \node[smallstate] (h16) [below = of h15] {};

  \node[smallstate] (2) [above =  of 14]{};
  \node[smallstate] (h22) [above left =0.6 of 2] {};
  \node[smallstate] (h23) [above = of h22] {};
  \node[smallstate] (h24) [above right = 0.6 of h23] {};
  \node[smallstate] (h25) [below right = 0.6 of h24] {};
  \node[smallstate] (h26) [below = of h25] {};

  \node[smallstate] (-11) [left = of 01] {};
  \node[smallstate] (-12) [left = of -11] {};
  \node[smallstate] (-13) [left = of -12] {};
  \node[smallstate] (-1) [above = of -13] {};
  \node[smallstate] (h-12) [above left =0.6 of -1] {};
  \node[smallstate] (h-13) [above = of h-12] {};
  \node[smallstate] (h-14) [above right = 0.6 of h-13] {};
  \node[smallstate] (h-15) [below right = 0.6 of h-14] {};
  \node[smallstate] (h-16) [below = of h-15] {};

  \node[smallstate] (-21) [left = of -13] {};
  \node[smallstate] (-22) [left = of -21] {};
  \node[smallstate] (-23) [left = of -22] {};
  \node[smallstate] (-2) [above = of -23] {};
  \node[smallstate] (h-22) [above left =0.6 of -2] {};
  \node[smallstate] (h-23) [above = of h-22] {};
  \node[smallstate] (h-24) [above right = 0.6 of h-23] {};
  \node[smallstate] (h-25) [below right = 0.6 of h-24] {};
  \node[smallstate] (h-26) [below = of h-25] {};

  \node[] (dr) [right = 1 of 14]{$\cdots$};
  \node[] (dl) [left = 1 of -23]{$\cdots$};


    \path[->]
        (0) edge[blue]
            (01);

    \path[->]
        (01) edge[red]
            (02);

    \path[->]
        (02) edge[red]
            (03);

    \path[->]
        (03) edge[red]
            (04);

    \path[->]
        (1) edge[blue]
            (04);

    \path[->]
        (0) edge[red]
            (h02);

    \path[->]
        (h02) edge[red]
            (h03);

    \path[->]
        (h03) edge[red]
            (h04);

    \path[->]
        (h04) edge[red,thick]
            node[pos=0.7,above right=-.5mm] {$\scriptstyle{(6,6)}$}
            (h05);

    \path[->]
        (h05) edge[red]
            (h06);

    \path[->]
        (h06) edge[red]
            (0);

    \path[->]
        (04) edge[red]
            (12);

    \path[->]
        (12) edge[red]
            (13);

    \path[->]
        (13) edge[red]
            (14);

    \path[->]
        (2) edge[blue]
            (14);

    \path[->]
        (1) edge[red]
            (h12);

    \path[->]
        (h12) edge[red]
            (h13);

    \path[->]
        (h13) edge[red]
            (h14);

    \path[->]
        (h14) edge[red,thick]
             node[pos=0.7,above right=-.5mm] {$\scriptstyle{(6,6)}$}
            (h15);

    \path[->]
        (h15) edge[red]
            (h16);

    \path[->]
        (h16) edge[red]
            (1);

    \path[->]
        (2) edge[red]
            (h22);

    \path[->]
        (h22) edge[red]
            (h23);

    \path[->]
        (h23) edge[red]
            (h24);

    \path[->]
        (h24) edge[red,thick]
             node[pos=0.7,above right=-.5mm] {$\scriptstyle{(6,6)}$}
            (h25);

    \path[->]
        (h25) edge[red]
            (h26);

    \path[->]
        (h26) edge[red]
            (2);

    \path[->]
        (-13) edge[red]
            (-12);

    \path[->]
        (-12) edge[red]
            (-11);

    \path[->]
        (-11) edge[red]
            (01);

    \path[->]
        (-1) edge[blue]
            (-13);

    \path[->]
        (-1) edge[red]
            (h-12);

    \path[->]
        (h-12) edge[red]
            (h-13);

    \path[->]
        (h-13) edge[red]
            (h-14);

    \path[->]
        (h-14) edge[red,thick]
             node[pos=0.7,above right=-.5mm] {$\scriptstyle{(6,6)}$}
            (h-15);

    \path[->]
        (h-15) edge[red]
            (h-16);

    \path[->]
        (h-16) edge[red]
            (-1);

    \path[->]
        (-23) edge[red]
            (-22);

    \path[->]
        (-22) edge[red]
            (-21);

    \path[->]
        (-21) edge[red]
            (-13);

    \path[->]
        (-2) edge[blue]
            (-23);

    \path[->]
        (-2) edge[red]
            (h-22);

    \path[->]
        (h-22) edge[red]
            (h-23);

    \path[->]
        (h-23) edge[red]
            (h-24);

    \path[->]
        (h-24) edge[red,thick]
             node[pos=0.7,above right=-.5mm] {$\scriptstyle{(6,6)}$}
            (h-25);

    \path[->]
        (h-25) edge[red]
            (h-26);

    \path[->]
        (h-26) edge[red]
            (-2);

    \path[->]
        (14) edge[red]
            (dr);

\path[->]
        (dl) edge[red]
            (-23);

\end{tikzpicture}
\caption{Stallings automaton for $\HH_1 \cap \HH_2 $ (Case 2)}
\label{fig: Stallings intersection example2}
\end{figure}
The corresponding (infinite) basis
for $\HH_1 \cap \HH_2 $
is
$\set{
y x^{3k} y^{-1} x^6 y  x^{-3k} y^{-1} \, \T^{(6,6)}
\st k \in \ZZ
}$.

\goodbreak

\paragraph{\large{Case 3.}}
Let $\vect{a} = (3,3),\,\vect{d}=(2,2) \in \ZZ^2$, and $L_1 = \gen{(2,2)},\,L_2=\gen{(0,0)} \leqslant \ZZ^2$.

In this case,
 $\matr{D} =
      \left(
      \begin{smallmatrix}
      0 & 0\\3&\hfill 3
      \end{smallmatrix}
      \right)$,
      ${\matr{M} =
      \left(
      \begin{smallmatrix*}[r]
      1 & 0\\0& 2
      \end{smallmatrix*}
      \right)}
      $,
      $
      \matr{P} = \matr{Q} =
      \left(
      \begin{smallmatrix*}[r]
      1 & 0\\0& 1
      \end{smallmatrix*}
      \right)
      $, and
      $\matr{S} =
      \left(
      \begin{smallmatrix*}[r]
      1 & 0\\0& 2
      \end{smallmatrix*}
      \right)
      $.
  Therefore,
  $\stallings{(\HH_1 \cap \HH_2) \prF, \set{w_1,w_2}} \isom
  \Cayley{\ZZ / \ZZ \oplus \ZZ / 2\ZZ, \set{(1,0),(0,1)}} \isom
  \Cayley{\ZZ/2\ZZ \,, \set{0,1}}$,
  which takes the form:
    \begin{figure}[H]
    \centering
      \begin{tikzpicture}[shorten >=1pt, node distance=1.2 and 1.6, on grid,auto,>=stealth']
      \node[state,accepting] (0) {};
      \node[state] (1) [right = of 0]{};

        \path[->]
            (0) edge[violet,loop,min distance=15mm, out=225,in=135]
                node[left=0.01] {$w_1$}
                (0);

        \path[->]
            (1) edge[violet,loop,min distance=15mm, out=45,in=-45]
                (1);

        \path[->]
            (0) edge[Green,bend left = 30]
                node[above=0.01] {$w_2$}
                (1);

        \path[->]
            (1) edge[Green,bend left = 30]
                (0);

    \end{tikzpicture}
    \vspace{-15pt}
    \caption{Stallings automaton corresponding to $\Cayley{\ZZ/2\ZZ,  \set{0,1}}$}
    \end{figure}
    After replacing the arcs reading
    $w_1$ and $w_2$ with the enriched paths reading
    $
    x^6 \, \T^{(6,6),(6,6)}$ and
     $
     y x^3 y^{-1}\, \T^{(3,3),(0,0)}$,
    folding, normalizing (\wrt the spanning tree having as cyclomatic arcs the thicker ones), and equalizing, we obtain the Stallings automaton:
\begin{figure}[H]
\centering
  \begin{tikzpicture}[shorten >=1pt, node distance=1 and 1, on grid,auto,>=stealth']
     \node[state,accepting] (0) {};

  \node[smallstate] (h02) [below left =0.9 of 0] {};
  \node[smallstate] (h03) [left = of h02] {};
  \node[smallstate] (h04) [above left = 0.9 of h03] {};
  \node[smallstate] (h05) [above right = 0.9 of h04] {};
  \node[smallstate] (h06) [right = of h05] {};

  \node[smallstate] (h11) [right = of 0] {};
  \node[smallstate] (h12) [above right =0.9 of h11] {};
  \node[smallstate] (h13) [right = of h12] {};
  \node[smallstate] (h14) [below right = 0.9 of h13] {};
  \node[smallstate] (h15) [below left = 0.9 of h14] {};
  \node[smallstate] (h16) [left = of h15] {};

  \node[smallstate] (2) [right = of h14] {};
  \node[smallstate] (h22) [above right =0.9 of 2] {};
  \node[smallstate] (h23) [right = of h22] {};
  \node[smallstate] (h24) [below right = 0.9 of h23] {};
  \node[smallstate] (h25) [below left = 0.9 of h24] {};
  \node[smallstate] (h26) [left = of h25] {};

  \path[->]
        (0) edge[blue]
            (h11);

  \path[->]
        (0) edge[red]
            (h02);

    \path[->]
        (h02) edge[red]
            (h03);

    \path[->]
        (h03) edge[red]
            (h04);

    \path[->]
        (h04) edge[red]
            (h05);

    \path[->]
        (h05) edge[red]
            (h06);

    \path[->]
        (h06) edge[red,thick]
            node[pos=0.7,above right =-.2mm] {$\scriptstyle{(6,6)}$}
            (0);

      \path[->]
        (h11) edge[red]
            (h12);

    \path[->]
        (h12) edge[red]
            (h13);

    \path[->]
        (h13) edge[red]
            (h14);

    \path[->]
        (h14) edge[red]
            (h15);

    \path[->]
        (h15) edge[red]
            (h16);

    \path[->]
        (h16) edge[red,thick]
            (h11);

  \path[->]
        (0) edge[red]
            (h02);

    \path[->]
        (2) edge[red]
            (h22);

    \path[->]
        (h22) edge[red]
            (h23);

    \path[->]
        (h23) edge[red]
            (h24);

    \path[->]
        (h24) edge[red]
            (h25);

    \path[->]
        (h25) edge[red]
            (h26);

    \path[->]
        (h26) edge[red,thick]
            node[pos=0.7,below left=-.2mm] {$\scriptstyle{(6,6)}$}
            (2);

  \path[->]
        (2) edge[blue]
            (h14);

\end{tikzpicture}
\caption{Stallings automaton for $\HH_1 \cap \HH_2 $ (Case 3)}
\label{fig: Stallings interesection example3}
\end{figure}
This provides the basis
$\set{
x^6   \T^{(6,6)},
y  x^6  y^{-1},
y  x^3  y^{-1}  x^6  y  x^{-3}  y^{-1} \T^{(6,6)}
}$
 for ${\HH_1 \cap \HH_2}$.

\paragraph{\large{Case 4.}}

Let $\vect{a} = (3,3),\,\vect{d}=(2,2) \in \ZZ^2$, and $L_1 = \gen{(1,1)},\,L_2=\gen{(0,0)} \leqslant \ZZ^2$.

In this case, 
$\matr{D} =
      \left(
      \begin{smallmatrix}
      0 & 0\\3&\hfill 3
      \end{smallmatrix}
      \right)$,
      and
      ${\matr{M} = \matr{P} = \matr{Q} = \matr{S} =
      \left(
      \begin{smallmatrix*}[r]
      1 & 0\\0& 1
      \end{smallmatrix*}
      \right)}
      $.
  Therefore,
  \[
  \stallings{(\HH_1 \cap \HH_2) \prF, \set{w_1,w_2}} \,\isom\,
\Cayley{\ZZ / \ZZ \oplus \ZZ/\ZZ \,,\, \set{(1,0),(0,1)}} \,\isom\,
  \cayley{\set{0} \,,\, \mset{0,0}},
  \]
which takes the form:
  \vspace{-25pt}
\begin{figure}[H]
\centering
  \begin{tikzpicture}[shorten >=1pt, node distance=1.2 and 1.6, on grid,auto,>=stealth']
  \node[state,accepting] (0) {};

  \path[->]
        (0) edge[Green,loop,min distance=19mm, in=225,out=135]
            node[left=0.01] {$w_1$}
            (0);

    \path[->]
        (0) edge[violet,loop,min distance=19mm, in=45,out=-45]
            node[right=0.01] {$w_2$}
            (0);
\end{tikzpicture}
\vspace{-30pt}
\caption{Stallings automaton corresponding to $\Cayley{\set{0} \,,\, \mset{0,0}}$}
\end{figure}

Recall that we are using Cayley \emph{multi}digraphs. Hence, esoteric objects like $\Cayley{\set{0} \,,\, \mset{0,0}}$ (the Cayley multidigraph of the trivial group \wrt the trivial generator \emph{considered twice}) may appear from our construction.

 After replacing the arcs reading
  $w_1$ and $w_2$ with the enriched paths reading
  $
  x^6 \, \T^{(6,6),(6,6)}$ and
  $
  y x^3 y^{-1}\, \T^{(3,3),(0,0)}$,
folding, normalizing
(\wrt the spanning tree having as cyclomatic arcs the thicker ones), and equalizing, we obtain the Stallings automaton:
\begin{figure}[H]
\centering
\begin{tikzpicture}[shorten >=1pt, node distance=1cm and 1cm, on grid,auto,>=stealth']
  \node[state,accepting] (0'0) {};
  \node[state] (1'1) [below left of = 0'0] {};
  \node[state] (2'0) [left of = 1'1] {};
  \node[state] (0'1) [above left  of = 2'0] {};
  \node[state] (1'0) [above right of = 0'1] {};
  \node[state] (2'1) [right of = 1'0] {};

  \node[state] (2'2) [right = of 0'0] {};
  \node[state] (0'2) [below right of = 2'2] {};
  \node[state] (1'2) [above right of = 2'2] {};


    \path[->]
        (1'1) edge[red,thick]
            node[pos=0.7,below right = -0.1] {$\scriptstyle{(6,6)}$}
            (0'0);

    \path[->]
        (1'0) edge[red]
            (0'1);

    \path[->]
        (2'1) edge[red]
            (1'0);

    \path[->]
        (2'0) edge[red]
            (1'1);

    \path[->]
        (0'1) edge[red]
            (2'0);

    \path[->]
         (0'0) edge[red]
            (2'1);

    \path[->]
        (0'0) edge[blue]
            (2'2);

    \path[->]
        (0'2) edge[red]
            (1'2);

    \path[->]
        (1'2) edge[red,thick]
            (2'2);

    \path[->]
        (2'2) edge[red]
            (0'2);
\vspace{5pt}
\end{tikzpicture}
\caption{Stallings automaton for $\HH_1 \cap \HH_2$ (Case 4)}
\label{fig: Stallings intersection example4}
\end{figure}
This provides the basis
$\set{ \,
x^6 \, \T^{(6,6)},
y  x^3  y^{-1}
\,}$
 for $\HH_1 \cap \HH_2$.

\begin{rem}
Comparing the Cases 2, 3 and 4, we see that a slight change in one of the abelian parts can seriously affect the behavior of the intersection.
\end{rem}

\goodbreak

\paragraph{\large{Case 5.}} \label{case: 5}

Let $\vect{a} = (6,6),\,\vect{d}=(4,4)$, and $L_1 = \gen{(6p,6p)}$ ($0 \neq p\in\ZZ$), ${L_2=\gen{(0,0)} \leqslant \ZZ^2}$.

 In this case,
 $\matr{D} 
      =
      \left(
      \begin{smallmatrix}
      0 & 0\\6&\hfill 6
      \end{smallmatrix}
      \right)$,
      $\matr{M} =
      \left(
      \begin{smallmatrix*}[r]
      1 & 0\\0& p
      \end{smallmatrix*}
      \right)
      $,
 $\matr{P} =
  \matr{Q} =
  \left(
  \begin{smallmatrix*}[r]
    1 & 0 \\
    0 & 1
  \end{smallmatrix*}
  \right)
  $,
  and
  $\matr{S} =
  \left(
  \begin{smallmatrix*}[r]
    1 & 0 \\
    0 & p
  \end{smallmatrix*}
  \right)
  $.
Therefore,
  $
  \Stallings{(\HH_1 \cap \HH_2) \prF, \set{w_1,w_2}}
  \isom
  \Cayley{\ZZ / \ZZ \oplus \ZZ / p\ZZ, \set{(1,0),(0,1)}} \isom
  \Cayley{\ZZ/p\ZZ \,,\, \set{0,1} }
  $:
\begin{figure}[H]
\centering
\begin{tikzpicture}[shorten >=0.5pt, node distance=1 and 1.1, on grid,auto,>=stealth']
\def \p {8}
\def \m {18}
\def \o {6}
\def \R {1.3cm}
\def \r {1.5cm}

\pgfmathsetmacro{\pp}{int(\p-1)}

\def \margin {2} 

\node[state,accepting] (0) at (1.3,0) {};

\foreach \s in {1,...,\pp}
{
  \node[state] (\s) at ({360/\p * (\s)}:\R) {};

}

\foreach \s in {1,...,\pp}
{
  \pgfmathsetmacro{\ss}{int(\s-1)}
  \path[->]
        (\ss) edge[Green]
            (\s);
}

\path[->]
        (\pp) edge[gray,dashed]
            node[below right = -0.1 and 0 ,black] {$p$ \text{ vertices}}
            (0);

\foreach \s in {1,...,\pp}
{
  \path[->]
        (\s) edge[violet,loop,min distance=9mm, in=\s*360/\p+45,out=\s*360/\p-45]
            (\s);
}

  \path[->]
        (0) edge[violet,loop,min distance=11mm, in=45,out=-45]
            (0);
\end{tikzpicture}
\vspace{-7pt}
\caption{Stallings automaton corresponding to $\cayley{\ZZ/p\ZZ,\set{0,1}}$}
\label{fig: }
\end{figure}
 After replacing the arcs reading
  $w_1$ and $w_2$ with the appropriate enriched paths,
folding, normalizing
(\wrt the spanning tree having as cyclomatic arcs the thicker ones), and equalizing,
we obtain the Stallings automaton:
\begin{figure}[H]
\centering
\begin{tikzpicture}[shorten >=0.5pt, node distance=1 and 1.2, on grid,auto,>=stealth']
\def \p {8}
\def \m {18}
\def \o {6}
\def \R {2cm}
\def \r {2.5cm}
\def \rr {2.8cm}
\def \rrr {3.2cm}
\def \rrrr {3.4cm}

\pgfmathsetmacro{\pp}{int(\p-1)}
\pgfmathsetmacro{\ppp}{int(\p-2)}

\def \margin {2} 

\node[state,accepting] (0) at (45:\r) {};

\foreach \s in {0,...,\pp}
{
  \node[smallstate] (\s) at ({360/\p * (\s)}:\R) {};
  \node[smallstate] (h\s1) at ({360/\p * (\s)}:\r) {};
  \node[smallstate] (h\s2) at ({360/\p * (\s) - 8} :\rr) {};
  \node[smallstate] (h\s6) at ({360/\p * (\s) + 8} :\rr) {};
  \node[smallstate] (h\s3) at ({360/\p * (\s) - 7} :\rrr) {};
  \node[smallstate] (h\s5) at ({360/\p * (\s) + 7} :\rrr) {};
  \node[smallstate] (h\s4) at ({360/\p * (\s)} :\rrrr) {};

}

\foreach \s in {0,...,\ppp}
{

  \foreach \k in {1,...,3}{
    \node[smallstate] (\s\k) at ({360/\p * (\s) + 120/\p * \k}:\R) {};
  }
}

\foreach \s in {1,...,\pp}
{
  \pgfmathsetmacro{\ss}{int(\s-1)}
  \path[->]
        (\ss) edge[red]
            (\ss1);

    \path[->]
        (\ss1) edge[red]
            (\ss2);

    \path[->]
        (\ss2) edge[red]
            (\ss3);
}

\path[->]
        (\pp) edge[gray,dashed]
            node[below right] {($p$ \text{times})}
            (0);

\foreach \s in {0,...,\pp}
{
    \path[->]
        (h\s1) edge[blue]
            (\s);

    \path[->]
        (h\s1) edge[red]
            (h\s2);

    \path[->]
        (h\s2) edge[red]
            (h\s3);

    \path[->]
        (h\s3) edge[red]
            (h\s4);

    \path[->]
        (h\s4) edge[red,thick]
            (h\s5);

    \path[->]
        (h\s5) edge[red]
            (h\s6);

    \path[->]
        (h\s6) edge[red]
            (h\s1);
}

%
%
%

\node[red] (u1) at ({360/\p * 0 + 5}:3.7){$\scriptscriptstyle{(12,12)}$};
\node[red] (u2) at ({360/\p * 1 + 3}:3.55){$\scriptscriptstyle{(12,12)}$};
\node[red] (u3) at ({360/\p * 2 + 11}:3.5){$\scriptscriptstyle{(12,12)}$};
\node[red] (u4) at ({360/\p * 3 + 8}:3.65){$\scriptscriptstyle{(12,12)}$};
\node[red] (u5) at ({360/\p * 4 + 5}:3.7){$\scriptscriptstyle{(12,12)}$};
\node[red] (u6) at ({360/\p * 5 + 4}:3.55){$\scriptscriptstyle{(12,12)}$};
\node[red] (u7) at ({360/\p * 6 + 11}:3.5){$\scriptscriptstyle{(12,12)}$};
\node[red] (u8) at ({360/\p * 7 + 8}:3.65){$\scriptscriptstyle{(12,12)}$};

\path[->]
        (02) edge[red,thick]
            (03);

\end{tikzpicture}
\caption{Stallings automaton for $\HH_1 \cap \HH_2$ (Case 5)}
\label{fig: Stallings intersection example6}
\end{figure}
This provides the basis
$
\set{ \,
y x^{3p} y^{-1}
}
\cup
\set{
y x^{3k} y^{-1} x^6 y x^{-3k} y^{-1}\, \T^{(12,12)}
\st
k \in [0,p-1]
\,}
$
for ${\HH_1 \cap \HH_2}$.

\goodbreak

\begin{rem}
Case 5 above points out the following interesting consequence:
not only the intersection of two finitely generated subgroups $ \HH_1, \HH_2 \leqslant \FTA$ can be of infinite rank,
but \emph{even when it is finitely generated}, one can no longer bound the rank of $\HH_1 \cap \HH_2$ in terms of the ranks of the intersecting subgroups. This fact is relevant because it denies any possible extension of the recently proved Hanna Neumann conjecture to groups containing $\Free[2] \times \ZZ$.

Indeed, $\HH_1 \,=\, \gen{\,  x^3 \T^{(6,6)}, yx, y^3 x y^{-2}, \T^{(6p,6p)}}$
 and  $\HH_2 \,=\, \gen{\,x^2 \T^{(4,4)},yxy^{-1}\,}$
are subgroups of
$ \Free[2] \times \gen{(1,1)}\leqslant  \Free[2] \times \ZZ^2$ of ranks $4$ and $2$ respectively (independently from $p$), whereas the intersection $\HH_1 \cap \HH_2$ has rank $p+1$. Moreover, note that by \Cref{rem: neglected parts of product} we can remove $y^3xy^{-2}$ from $\HH_1$ without affecting the intersection; this way we obtain two subgroups of~$\Free[2] \times \ZZ$ of ranks $3$ and $2$ whose intersection has rank $p+1$.

Note that this is the minimum possible sum of ranks for such an example: if one of the intersecting subgroups has rank $1$, then the intersection must be cyclic; if one of the intersecting subgroups is abelian then the intersection has rank at most $2$. It only remains to consider the case of two subgroups of rank $2$ with trivial abelian part.
But then, by \Cref{rem: L1=L2=0},  $\HH_1 \cap \HH_2$ is either non finitely generated or $(\HH_1 \cap \HH_2)\prF = \HH_1 \prF \cap \HH_2 \prF$, and hence has rank bounded by $1(2-1) (2-1) + 1 = 2$. So, the minimum possible ranks of subgroups $\HH_1,\HH_2 \leqslant \Free[2] \times \ZZ$ with intersection of arbitrarily large finite rank are $3$ and $2$, as claimed.
\end{rem}

\section*{Acknowledgements}

The first author was partially supported by CMUP (UID/MAT/00144/2019), which is funded by FCT (Portugal) with national (MCTES) and European structural funds through the programs FEDER, under the partnership agreement PT2020.
Parts of this project were developed during the participation of the first author in the
``Logic and Algorithms in Group Theory'' meeting
held in the
Hausdorff Research Institute for Mathematics (Bonn) in fall 2018.

Both authors acknowledge partial support from the Spanish Agencia Estatal de Investigación, through grant MTM2017-82740-P (AEI/FEDER, UE), and also from the Barcelona Graduate School of Mathematics through the “María de Maeztu” Programme for Units of Excellence in R\&D (MDM-2014-0445).

The work was supported in part by MINECO grant PID2019-107444GA-I00 and the Basque Government grant IT974-16.

\renewcommand*{\bibfont}{\small}
\printbibliography

\Addresses

\end{document}